\documentclass[a4paper, twoside]{scrartcl}
\usepackage{etex}
\usepackage[a4paper, top=88pt, bottom=88pt, left=72pt, right=72pt, headsep=16pt, footskip=28pt]{geometry}
\usepackage{amsmath, amssymb, amsthm, xcolor, tabularx, graphics}

\usepackage[hyphens]{url}
\usepackage{hyperref}
\hypersetup{
	breaklinks=true,
	colorlinks=true,
	linkcolor=blue,
	citecolor=blue,
	urlcolor=blue,
}
\usepackage[nameinlink]{cleveref}

\usepackage{enumitem, moreenum}
\setlist{nosep}
\usepackage{mleftright} \mleftright
\renewcommand{\qedsymbol}{$\blacksquare$}
\renewenvironment{proof}[1][\proofname]{\noindent{\bfseries\sffamily #1.} }{\hfill\qedsymbol\medskip}
\usepackage[labelfont={sf,bf}]{caption}
\usepackage{scrlayer-scrpage, xhfill}
\automark[section]{section}
\setkomafont{pageheadfoot}{\normalcolor\sffamily}
\setkomafont{pagenumber}{\normalfont\normalsize\sffamily}
\clearpairofpagestyles
\let\oldtitle\title
\renewcommand{\title}[1]{\oldtitle{#1}\newcommand{\theshorttitle}{#1}}
\newcommand{\shorttitle}[1]{\renewcommand{\theshorttitle}{#1}}
\let\oldauthor\author
\renewcommand{\author}[1]{\oldauthor{#1}\newcommand{\theshortauthor}{#1}}
\newcommand{\shortauthor}[1]{\renewcommand{\theshortauthor}{#1}}
\cohead{\xrfill[0.525ex]{0.6pt}~\theshorttitle~\xrfill[0.525ex]{0.6pt}}
\cehead{\xrfill[0.525ex]{0.6pt}~\theshortauthor~\xrfill[0.525ex]{0.6pt}}
\newcommand{\theabstract}[1]{\par\bgroup\noindent\textbf{\textsf{Abstract.}} #1\egroup}
\newcommand{\thekeywords}[1]{\par\smallskip\bgroup\noindent\textbf{\textsf{Keywords.}}\newcommand{\and}{ $\bullet$ } #1\egroup}
\newcommand{\themsc}[1]{\par\smallskip\bgroup\noindent\textbf{\textsf{2020 Mathematics Subject Classification.}}\newcommand{\and}{ $\bullet$ } #1\egroup}

\newcommand*{\affilref}[1]{\ref{affiliation#1}}
\newcommand*{\affiliation}[3]{
	\footnotetext[#1]{\label{affiliation#2}#3}
}

\usepackage{natbib}
\setcitestyle{round}
\newcommand*{\arXiv}[1]{\bgroup\color{blue}\href{https://arxiv.org/abs/#1}{arXiv:#1}\egroup}
\newcommand*{\doi}[1]{\bgroup\color{blue}\href{https://dx.doi.org/#1}{doi:#1}\egroup}
\newcommand*{\email}[1]{\bgroup\color{blue}\href{mailto:#1}{#1}\egroup}
\renewcommand*{\url}[1]{\bgroup\color{blue}\href{#1}{#1}\egroup}
\newcommand{\todo}[1]{\bgroup\color{red}\bfseries#1\egroup}

\usepackage{mathtools}
\newcommand*{\ACG}{\mathrm{ACG}}

\newcommand*{\Borel}[1]{\mathcal{B}(#1)}

\newcommand*{\defeq}{\coloneqq}
\newcommand*{\qefed}{\eqqcolon}

\newcommand*{\Hausdorff}{\mathcal{H}}

\newcommand*{\Naturals}{\mathbb{N}}
\newcommand*{\Normal}{\mathrm{N}}
\newcommand*{\Gam}{\mathrm{Gam}}
\usepackage{dsfont}
\newcommand*{\one}{\mathds{1}}
\newcommand*{\probson}[1]{\mathcal{P}(#1)}

\newcommand*{\PropKernel}{Q}
\newcommand*{\quark}{\setbox0\hbox{$x$}\hbox to\wd0{\hss$\cdot$\hss}}
\newcommand*{\rd}{\mathrm{d}}

\newcommand*{\RadProjToSphere}{\Pi^{\Sphere}}
\newcommand*{\Reals}{\mathbb{R}}
\newcommand*{\Sphere}{\mathbb{S}}

\newcommand*{\TV}{\mathrm{TV}}
\newcommand*{\Uniform}{\mathrm{U}}
\newcommand*{\XX}{\mathbb{X}}
\newcommand*{\YY}{\mathbb{Y}}
\newcommand*{\HH}{\mathbb{H}}
\newcommand*{\MM}{\mathbb{M}}

\newcommand*{\U}{\mathcal{U}}
\newcommand*{\V}{\mathcal{V}}

\renewcommand*{\geq}{\geqslant}

\renewcommand*{\leq}{\leqslant}

\newcommand*{\iidsim}{\mathrel{\stackrel{\textsc{iid}}{\sim}}}

\newcommand*{\Prob}{\mathbb{P}}
\DeclareMathOperator{\rg}{ran}

\newcommand{\widebar}{\bar}
\newcommand*{\xx}{\widebar{x}}
\newcommand*{\yy}{\widebar{y}}
\newcommand*{\KK}{\widebar{K}}
\newcommand*{\PPhi}{\widebar{\Phi}}
\newcommand*{\ev}[1]{\mathbb{E}\left[ #1 \right]}

\newcommand*{\e}{\mathrm{e}}
\DeclareMathOperator{\gap}{gap}

\newcommand*{\norm}[1]{\Vert #1 \Vert}

\newcommand*{\innerprod}[2]{\langle #1 , #2 \rangle}
\newcommand*{\absval}[1]{\vert #1 \vert}
\newcommand*{\set}[2]{\{ #1 \mid #2 \}}
\newcommand*{\bignorm}[1]{\bigl\Vert #1 \bigr\Vert}

\newcommand*{\Norm}[1]{\left\Vert #1 \right\Vert}

\newcommand*{\Absval}[1]{\left\vert #1 \right\vert}
\newcommand*{\Set}[2]{\left\{ #1 \,\middle\vert\, #2 \right\}}

\DeclareMathOperator*{\sspan}{span}
\newcommand*{\Span}[1]{\sspan \left\{ #1 \right\}}

\newtheorem{theorem}{\sffamily Theorem}[section]
\newtheorem{corollary}[theorem]{\sffamily Corollary}
\newtheorem{lemma}[theorem]{\sffamily Lemma}
\newtheorem{proposition}[theorem]{\sffamily Proposition}
\theoremstyle{definition}

\newtheorem{remark}[theorem]{\sffamily Remark}
\newtheorem{myAlgorithm}[theorem]{\sffamily Transition Mechanism}

\crefname{myAlgorithm}{Transition Mechanism}{Transition Mechanisms}
\Crefname{myAlgorithm}{Transition Mechanism}{Transition Mechanisms}

\numberwithin{equation}{section}
\numberwithin{figure}{section}
\numberwithin{table}{section}

\usepackage{acro}
\DeclareAcronym{ACG}{short=ACG, long=angular central Gaussian}
\DeclareAcronym{BIP}{short=BIP, long=Bayesian inverse problem}
\DeclareAcronym{DFG}{short=DFG, long=Deutsche For\-schungs\-gemein\-schaft}
\DeclareAcronym{BVP}{short=BVP, long=boundary value problem}
\DeclareAcronym{ESS}{short=ESS, long=elliptical slice sampler}
\DeclareAcronym{MCMC}{short=MCMC, long=Markov chain Monte Carlo}
\DeclareAcronym{MH}{short=MH, long=Metropolis--Hastings}
\DeclareAcronym{pCN}{short=pCN, long=preconditioned Crank--Nicolson}
\DeclareAcronym{PDE}{short=PDE, long=partial differential equation}

\newcommand*{\wrt}{with respect to\ }

\usepackage{algorithm,algorithmic}

\begin{document}

\title{Dimension-independent\\Markov chain Monte Carlo on the sphere}
\shorttitle{Dimension-independent MCMC on the sphere}

\author{%
	H.\ C.\ Lie\textsuperscript{\affilref{Potsdam}}%
	\and
	D.\ Rudolf\textsuperscript{\affilref{Passau}}%
	\and
	B.\ Sprungk\textsuperscript{\affilref{Freiberg}}%
	\and
	T.\ J.\ Sullivan\textsuperscript{\affilref{Warwick},\affilref{Turing}}%
}
\shortauthor{H.~C.~Lie, D.~Rudolf, B.~Sprungk, and T.~J.~Sullivan}

\newcommand{\change}[1]{\bgroup\color{magenta}#1\egroup} 

\date{\today}

\maketitle

\affiliation{1}{Potsdam}{Institut f\"{u}r Mathematik, Universit\"{a}t Potsdam, Karl-Liebknecht-Stra{\ss}e 24--25, 14476 Potsdam OT Golm, \mbox{Germany} (\email{hanlie@uni-potsdam.de})}
\affiliation{2}{Passau}{Universit\"at Passau, Innstra{\ss}e 33, 94032 Passau, Germany (\email{daniel.rudolf@uni-passau.de})}
\affiliation{3}{Freiberg}{Technische Universit{\"a}t Bergakademie Freiberg,
09596 Freiberg, Germany (\email{bjoern.sprungk@math.tu-freiberg.de})}
\affiliation{4}{Warwick}{Mathematics Institute and School of Engineering, University of Warwick, Coventry, CV4 7AL, United Kingdom (\email{t.j.sullivan@warwick.ac.uk})}
\affiliation{5}{Turing}{Alan Turing Institute, 96 Euston Road, London, NW1 2DB, United Kingdom}

\begin{abstract}\small
	\theabstract{
	We consider Bayesian analysis on high-dimensional spheres with angular central Gaussian priors.
	These priors model antipodally symmetric directional data, are easily defined in Hilbert spaces and occur, for instance, in Bayesian binary classification and level set inversion.
	In this paper we derive efficient Markov chain Monte Carlo methods for approximate sampling of posteriors with respect to these priors.
	Our approaches rely on lifting the sampling problem to the ambient Hilbert space and exploit existing dimension-independent samplers in linear spaces.
    By a push-forward Markov kernel construction we then obtain Markov chains on the sphere which inherit reversibility and spectral gap properties from samplers in linear spaces.
    Moreover, our proposed algorithms show dimension-independent efficiency in numerical experiments.
	}

	\thekeywords{Markov chain Monte Carlo%
	\and%
	dimension independence%
	\and%
	directional statistics%
	\and%
	level set inversion%
	\and%
	high-dimensional manifolds}

	\themsc{
	{46T12}
	\and%
	{58C35}
	\and%
	{60J22}
	\and%
	{62H11}
	\and%
	{65C40}
	}
\end{abstract}

\section{Introduction}
\label{sec:introduction}
The \ac{MCMC} method is a standard tool for computational probability and recent years have seen increasing interest in \emph{dimension-independent} \ac{MCMC} schemes, i.e.\ those whose statistical efficiency and mixing rates do not degenerate to zero as the dimension of the sample space tends to infinity.
We mention here the \ac{pCN} scheme of \citet{CotterRobertsStuartWhite2013} --- see also \citep[Equation~(15)]{Neal1999} and \citep{Beskos_etal_2008} --- and the \ac{ESS} of \citet{Murray2010}, both of which rely on a Gaussian reference or prior measure.
Recently, the \ac{pCN} scheme has been combined with geometric \ac{MCMC} methods \citep{Beskos_etal_2017, RuSpru2018} and extended to classes of non-Gaussian priors \citep{Chen2018}, and for the \ac{ESS} geometric ergodicity was shown by \citet{NataElAl2021}.

In this work, we study whether these dimension-independent sampling schemes could also be modified for Bayesian analysis on high-dimensional manifolds.
As a starting point we focus on Bayesian inference on high-dimensional spheres \citep{Watson1983}.
We then consider the case of the unit sphere in a general Hilbert space.
We formulate some results in more generality, e.g.\ by replacing the ambient Hilbert space and the unit sphere with a pair of topological spaces that are related by a measurable mapping.
This allows us to extend some of our results to manifolds that are more general than the unit sphere.
Our choice of the sphere is further motivated by particular inverse problems on function spaces such as level set inversion --- more precisely, \emph{binary classification} --- where one is essentially interested only in recovering the pointwise \emph{sign} of a function $u \colon D \to \Reals$ on some domain $D$.
Thus, $u$ and $\alpha u$ for $\alpha > 0$ yield equivalent classifications and, hence, it is natural to consider the inverse problem just on some unit sphere of functions.

Previous works on \ac{MCMC} methods on manifolds --- such as those of \citet{BrubakerEtAl2012}, \citet{ByrneGirolami2013}, \citet{Diaconis_etal_2013}, \citet{MangoubiSmith2018}, and \citet{Zappa2018} --- derive algorithms which are based on the Hausdorff or surface measure as reference measure.
However, despite their use of geometric structure, the performance of such methods typically still degrades as the dimension of the sample space increases to infinity --- one reason being the degeneration of the target density \wrt the Hausdorff measure.

\subsection{Contribution}
\label{ssec:contribution}

In this paper, we aim to construct dimension-independent \ac{MCMC} methods in order to sample efficiently from target measures on high-dimensional spheres.
We identified the \ac{ACG} distribution as a suitable reference measure for this purpose.
The \ac{ACG} models antipodally symmetric directional data and is an alternative to the Bingham distribution \citep{Tyler1987}.
\ac{ACG} distributions and their mixtures have been applied in finite-dimensional directional-statistical problems such as geomagnetism \cite[Section~8]{Tyler1987}, imaging in neuroscience \citep{Tabelow2012}, and materials science \citep[Section~4]{Franke2016}.
\ac{ACG} distributions have been generalised to the projected normal distribution, for which the initial Gaussian distribution may have nonzero mean \citep{WangGelfand2013}.

The \ac{ACG} distribution is defined as the radial projection onto the sphere of a centred Gaussian measure on the ambient Hilbert space and thus yields a well-defined reference measure even in infinite-dimensional Hilbert spaces.
Moreover, the \ac{ACG} distribution can be applied in an acceptance-rejection method for sampling from several families of distributions on spheres and similar manifolds \citep{Kent2018}.
Thus, we anticipate that our proposed methods could also be exploited for dimension-independent \ac{MCMC} for posteriors with other priors, e.g.\ Bingham, Fisher--Bingham, or von Mises--Fisher priors.
However, we leave this question for future research, and focus on posteriors given \wrt the \ac{ACG} prior in this paper.

The particular structure of the \ac{ACG} prior allows us to lift the sampling problem to the ambient Hilbert space.
Thus, we can exploit existing dimension-independent \ac{MCMC} algorithms on linear spaces, e.g.\ the \ac{pCN} algorithm mentioned earlier.
In order to obtain Markov chains on the sphere, we use \emph{push-forward Markov kernels} as introduced by \citet{RuSpru2021}.
This approach then yields specific \ac{MCMC} algorithms that first draw from a suitable distribution a point on the ray defined by the current position on the sphere, and then take a step using a dimension-independent transition kernel.
The resulting state is finally ``reprojected'' to the sphere.

In summary, our contributions are as follows:
\begin{enumerate}[label=(\arabic*)]
	\item
	We propose two easily implementable \ac{MCMC} algorithms that generate reversible Markov chains on high-dimensional spheres, where the chains have as their invariant distribution a given posterior \wrt an ACG prior;

	\item
	We prove uniform ergodicity of the
	suggested Markov chains in finite-dimensional settings;

	\item
	We provide theoretical and numerical evidence for dimension-independent statistical efficiency of the proposed algorithms.
\end{enumerate}
Moreover, our numerical experiments show that some other existing \ac{MCMC} methods for sampling on manifolds --- see \Cref{ssec:related} for an overview --- exhibit decreasing statistical efficiency as the state space dimension increases.
Thus, we provide a first contribution to dimension-independent \ac{MCMC} on manifolds, and thereby demonstrate the feasibility of efficient Bayesian analysis on high-dimensional spheres.

\subsection{Outline}
\label{ssec:outline}

The remainder of this paper is structured as follows.
\Cref{ssec:related} overviews some related work in this area and \Cref{sec:preliminaries} sets out some basic notation.
In \Cref{sec:MCMC_HilbertSpace} we recall two basic \ac{MCMC} algorithms that are valid in infinite-dimensional Hilbert spaces.
Basic definitions and properties related to \ac{MCMC}, in particular the \ac{MH} and slice sampling paradigms, are provided in \Cref{sec:MCMC} for completeness.
In \Cref{sec:MCMC_on_Hilbert_sphere} we make our main theoretical contributions by developing a general framework for obtaining dimension-independent \ac{MCMC} methods on manifolds.
In particular, we derive and analyse two sampling methods on the sphere.
These methods are subjected to numerical tests, in the context of Bayesian binary classification and density estimation, in \Cref{sec:numerical_illustrations}.
Some closing remarks are given in \Cref{sec:closing}.
In the appendix we further recall some key facts about Gaussian and ACG measures (\Cref{sec:Gaussian_ACG}), describe related existing \ac{MCMC} algorithms on the sphere (\Cref{sec:GeodesicZappa}) and provide technical auxiliary results (\Cref{sec:image_of_MC_is_not_MC}).

\subsection{Overview of related work}
\label{ssec:related}

Classical references that treat statistical inference on the sphere include those of \citet{Watson1983}, who focusses exclusively on spheres in finite-dimensional Euclidean spaces, and \citet{MardiaJupp2000}, who focus on circular data but also treat spheres, Stiefel and Grassmann manifolds, and general manifolds.
Special manifolds such as the Stiefel and Grassmann manifolds have also been studied by \citet{Chikuse2003}.
A recent treatment that focuses on modern developments in directional statistics is given by \citet{LeyVerdebout2017}.
\citet{Srivastava2009} consider the infinite-dimensional unit sphere of diffeomorphisms in the context of computer vision, and then apply a Bayesian method for shape identification.
However, none of the cited works treat \ac{MCMC} sampling methods or Bayesian inference on high-dimensional manifolds.

Regarding sampling on embedded manifolds, Hamiltonian Monte Carlo methods are considered by \citet{BrubakerEtAl2012} and \citet{ByrneGirolami2013}, for instance, and \citet{Diaconis_etal_2013} propose a Gibbs sampler.
Moreover, \citet{MangoubiSmith2018} study the so-called ``geodesic walk'' algorithm and establish Wasserstein contraction under the assumption that the manifold has bounded, positive curvature.
The geodesic walk algorithm of \citet{MangoubiSmith2018} chooses a random element uniformly from the unit sphere in the tangent space and moves a fixed time or step size along the corresponding geodesic.
This could be used as a proposal in an \ac{MH} algorithm.
Similarly, \citet{Zappa2018} developed an \ac{MH} algorithm on manifolds where the proposed point is generated by a normally distributed tangential move into the ambient Euclidean space which is then suitably projected back to the manifold.
We will compare our algorithms particularly to that of \citet{Zappa2018} and the geodesic walk algorithm of \citet{MangoubiSmith2018}.
For the specific problem of designing \ac{MCMC} samplers on the sphere, \citet{LanZhouShahbaba2014} considered
Hamiltonian Monte Carlo for distributions that undergo several transformations in order to be defined on the unit sphere.
Their approach has been used by \citet{Holbrook2020} to perform Bayesian nonparametric density estimation based on the Bingham distribution as the prior.

We also mention the work of \citet{Yang2022}, which considers high-dimensional \ac{MCMC} methods for sampling from heavy-tailed distributions.
Their work uses stereographic projection to the sphere to prove desirable mixing properties for the resulting MCMC samplers as the dimension increases.
Two important differences between their work and our work are that they focus on sampling from heavy-tailed distributions on Euclidean spaces, while we consider sampling from the sphere in general Hilbert spaces and focus on the \ac{ACG} prior.

\subsection{Preliminaries and notation}
\label{sec:preliminaries}

Throughout, $(\Omega, \mathcal{A}, \Prob)$ will be a fixed probability space, which we assume to be rich enough to serve as a common domain of definition for all random variables under consideration.

Given a topological space $\XX$, $\probson{\XX}$ denotes the space of probability measures on the Borel $\sigma$-algebra $\Borel{\XX}$ of $\XX$.
Given another topological space $\YY$, $T_{\sharp} \mu \in \probson{\YY}$ denotes the \emph{push-forward} or \emph{image measure} of $\mu \in \probson{\XX}$ under a measurable map $T \colon \XX \to \YY$, i.e.
\begin{equation}
	\label{eq:push-forward}
	(T_{\sharp} \mu) (E) \defeq \mu ( T^{-1} (E) ) \equiv \mu ( \set{ x \in \XX }{ T(x) \in E } ) \quad \text{for each $E \in \Borel{\YY}$.}
\end{equation}
The range of a map $T$ is denoted $\rg(T)$.
Throughout this paper, we use `measurability' to refer to Borel measurability of a mapping between topological spaces or Borel measurability of a subset.

The absolute continuity of one measure $\mu \in \probson{\XX}$ with respect to another measure $\nu$ will be denoted by $\mu \ll \nu$.

We denote the $s$-dimensional Hausdorff measure by $\Hausdorff^{s}$.
If $E$ is an $s$-dimensional measurable set, then $\Hausdorff_{E}$ denotes the restriction of $\Hausdorff^{s}$ to $E$.

We denote the uniform distribution on a bounded subset $G\subset \Reals^d$ by $\Uniform[G]$ and the normal distribution with mean element $m$ and covariance operator $C$ by $\Normal(m,C)$.
For the convenience of the reader we provide a short overview of Gaussian measures on Hilbert spaces in \Cref{ssec:Gaussian}.

\section{MCMC in Hilbert spaces}\label{sec:MCMC_HilbertSpace}

We consider the case in which $\XX$ is a separable Hilbert space $\HH$ and the target or posterior distribution $\nu \in \probson{\HH}$ is determined by a density \wrt a mean-zero Gaussian reference or prior measure $\nu_{0} = \Normal(0,C)$ with covariance operator $C$ via
\begin{equation}
	\label{eq:post0}
	\frac{\rd \nu}{\rd \nu_{0}}(x)
	\propto \exp( - \Phi(x)),
	\quad
	\nu_{0} \text{-a.e.\ }x\in\HH,
\end{equation}
with measurable $\Phi\colon \HH \to \Reals$ satisfying
\[
	\int_{\HH} \exp(-\Phi(x)) \, \nu_{0}(\rd x) < \infty.
\]
In this setting we state two popular approaches for generating $\nu$-reversible Markov chains $(X_k)_{k\in\mathbb N}$.
The first is the \ac{pCN}-\ac{MH} algorithm \citep{Neal1999,CotterRobertsStuartWhite2013}.

\begin{algorithm}
	\caption{\ac{pCN}-\ac{MH} algorithm on $\HH$} \label{alg:pCN_on_H}
	\begin{algorithmic}[1]
		\STATE \textbf{Given:} prior $\nu_0 = \Normal(0,C)$ and target $\nu$ as in \eqref{eq:post0}
		\STATE \textbf{Initial:} step size $s\in(0,1]$ and state $x_0 \in \HH$
		\FOR{$k \in\Naturals_0$}
		\STATE Draw a sample $w_k$ of $\Normal(0, C)$ and set $y_{k+1} \defeq  \sqrt{1-s^2} x_k + s w_k  $
		\STATE Compute $\alpha \defeq  \min\{1, \exp\left(\Phi(x_{k}) - \Phi(y_{k+1})\right)\}$
		\STATE Draw a sample $u$ of $\Uniform[0,1]$
		\IF{$u\leq \alpha$}
		\STATE Set $x_{k+1} = y_{k+1}$
		\ELSE
		\STATE Set $x_{k+1} = x_{k}$
		\ENDIF
		\ENDFOR
	\end{algorithmic}
\end{algorithm}

Here a possible new state $y_{k+1}$ of the Markov chain given the current state $X_k = x_k$ is drawn according to the \ac{pCN}-\ac{MH} proposal kernel $Q(x_k,\cdot)$ where
\[
	\PropKernel(x, \rd y) \defeq \Normal \Bigl( \sqrt{1-s^2} x, s^2 C \Bigr),
\]
with $s\in (0,1]$ denoting a step size parameter.
The state $y_{k+1}$ is accepted as the new state $X_{k+1}$ only with probability $\alpha(x_k,y_{k+1})$, where the acceptance probability function $\alpha$ is given by
\[
	\alpha(x,y) \defeq \min\left\{ 1, \exp(\Phi(x) - \Phi(y))\right\} ;
\]
otherwise, the Markov chain remains at $X_{k+1} \defeq x_k$.
\Cref{alg:pCN_on_H} describes how to realise a Markov chain with \ac{pCN}-\ac{MH} transition kernel.

Next, we consider the \ac{ESS} algorithm suggested by \citet{Murray2010}.
In this reference it is stated in a finite-dimensional setting, but the \ac{ESS} algorithm can be lifted
also to infinite-dimensional settings.
Given $X_k=x$ we first choose a slice $\HH_t \defeq \{x\in\HH\colon \exp(-\Phi(x)) \geq t\}$ at random by drawing $t$ according to $t\sim \Uniform[0,\exp(-\Phi(x_k))]$.
We then sample a new state $X_{k+1} = y$ where $y \in \HH_t$ according to the restriction of $\nu_0=\Normal(0,C)$ to $\HH_t$.
In order to achieve the second step in an approximate way, the \ac{ESS} employs a certain transition mechanism using randomly drawn ellipses in $\HH$ and a shrinkage procedure.
We state this transition in \Cref{alg:ell_shrink}, which we call \emph{shrink-ellipse}$(x,t)$.
Thus, the \ac{ESS} sampler is a hybrid slice sampler.
Its algorithmic realisation is described in \Cref{alg:ell_slice_sampling}.

\begin{algorithm}
\caption{\ac{ESS} algorithm on $\HH$}
\label{alg:ell_slice_sampling}
\begin{algorithmic}[1]
	\STATE \textbf{Given:} prior $\nu_0 = \Normal(0,C)$ and target $\nu$ as in \eqref{eq:post0}
	\STATE \textbf{Initial:} state $x_0 \in \HH$
	\FOR{$k \in\Naturals_0$}
	\STATE Draw sample $t\sim \Uniform[0,\exp(-\Phi(x_k))]$
	\STATE Set $x_{k+1}=\text{shrink-ellipse}(x_k,t)$ \; (see \Cref{alg:ell_shrink})
	\ENDFOR
\end{algorithmic}
\end{algorithm}

\begin{algorithm}
	\caption{Elliptical shrinkage ($\text{shrink-ellipse}(x,t)$)} \label{alg:ell_shrink}
	\begin{algorithmic}[1]
		\STATE \textbf{Input:} state $x\in\HH$ and level $t\in(0,\infty)$
		\STATE \textbf{Output:} state $y$ in level set $\HH_t$
		\STATE Draw a sample $w\sim\Normal(0,C)$
		\STATE Draw a sample $\theta\sim \Uniform[0,2\pi]$
		\STATE Set $\theta_{\min}=\theta-2\pi$ and $\theta_{\max}=\theta$
		\WHILE{$\exp(\Phi(y))<t$}
		\IF{$\theta<0$}
		\STATE Set $\theta_{\min} = \theta$
		\ELSE
		\STATE Set $\theta_{\max} = \theta$
		\ENDIF
		\STATE Draw a sample $\theta\sim \Uniform[\theta_{\min},\theta_{\max}]$
		\STATE Set $y=\cos(\theta) x+\sin(\theta) w$
		\ENDWHILE
	\end{algorithmic}
\end{algorithm}

It can be shown that the transition kernel of the \ac{ESS} sampler has $\nu$ as its invariant distribution; see \citep{Murray2010} and \citep[Section~3]{NatPhD21} for further details.
For a more comprehensive introduction to \ac{MCMC} and the \ac{MH} and slice sampling approaches, see \Cref{sec:MCMC}.

\begin{remark}
	As noted by \citet{Murray2010}, both the \ac{ESS} algorithm and the \ac{pCN} algorithm draw proposal states from ellipses that are accepted or rejected.
	In the \ac{pCN} algorithm, the random proposal $X'$ satisfies $X'=\sqrt{1-s^2}x+ s Z$, where $Z\sim\Normal(0,C)$.
	For a fixed realisation $z$ of $Z$ and for varying $s\in (0,1)$, the set $\set{ \sqrt{1-s^2}x+sz }{ s\in(0,1) }$ is half of the ellipse passing through $x$ and $z$ centred at the origin, since $\sqrt{1-s^2}x+sz=\cos(\theta)x+\sin(\theta)z$ for $\theta=\arcsin(s)$.
	In the elliptical slice sampling algorithm, a full ellipse instead of a half ellipse is used, thus providing a larger set of potential proposal states.
	Moreover, one never remains at the current state.
	Intuitively, using a larger set of potential proposal states might lead to faster convergence, as measured by the number of Markov chain steps.
\end{remark}

\section{Markov chain Monte Carlo on the sphere}
\label{sec:MCMC_on_Hilbert_sphere}

In this section, we construct and analyse \ac{MCMC} algorithms for approximate sampling from a probability distribution $\mu$ on a high-dimensional unit sphere $\Sphere^{d-1} \subset \Reals^{d}$ where $\mu$ admits a density with respect to an \ac{ACG} reference or prior measure $\mu_0$.
The \ac{ACG} measure is given as follows.
Consider the unit sphere $\Sphere \defeq \set{ x \in \HH}{ \norm{ x } = 1}$ of a separable and possibly infinite-dimensional Hilbert space $\HH$ as well as a centred Gaussian measure $\Normal(0, C)$ on $\HH$.
Furthermore, let $\RadProjToSphere\colon \HH \to \Sphere$ denote the radial projection to the sphere
\begin{equation}
	\label{eq:radial_projection_map_T}
	\RadProjToSphere\colon \HH \to\Sphere,\quad \RadProjToSphere(x) \defeq \begin{cases}
		\frac{x}{\norm{ x }}, & \text{if $x \neq 0$,}
		\\
		\bar{z}, & \text{if $x = 0$,}
	\end{cases}
\end{equation}
with a fixed but arbitrary $\bar{z}\in\Sphere$.
Then we call the probability measure
\[
	\mu_0 \defeq \RadProjToSphere_{\sharp} \Normal (0, C)
\]
the \emph{angular central Gaussian measure} with parameter $C$ and denote it by $\mu_0 = \ACG (C)$.
In the case where $\HH = \Reals^{d}$ with the usual Euclidean norm and $C\in\Reals^{d \times d}$ being symmetric and positive definite, one can show that the density $\rho \colon \Sphere^{d - 1} \to [0, \infty)$ of $\mu_0=\ACG (C)$ \wrt the $(d - 1)$-dimensional Hausdorff measure on the sphere is
\begin{align*}
	\rho(\xx)
	& = \frac{ \Gamma(d / 2) }{ 2 \pi^{d / 2} \sqrt{ \det C } } \ \norm{ \xx }_{C}^{- d} ;
\end{align*}
see \Cref{ssec:ACG} for details, including the definition \eqref{eq:precision_norm} of $\norm{ \xx }_{C}$.
We shall write bars over symbols to distinguish elements of $\Sphere$ from elements of $\HH$.
Thus, $x \in \HH$, while $\xx \in \Sphere$.

Consider a given target or posterior measure $\mu \in \probson{\Sphere}$ which is  absolutely continuous \wrt an \ac{ACG} reference or prior measure $\mu_{0} \defeq \ACG (C)$, i.e.,
\begin{equation}
	\label{eq:push_post}
	\frac{\rd \mu}{\rd \mu_{0}}(\xx)
	\propto \exp( - \PPhi(\xx)),
	\quad
	\xx \in \Sphere
\end{equation}
where $\PPhi\colon \Sphere \to \Reals$ denotes a measurable function that satisfies
\[
	\int_\Sphere \exp(-\PPhi(\xx)) \, \mu_{0}(\rd \xx) < \infty.
\]
The \ac{ACG} prior allows us to define an equivalent sampling problem in the ambient Hilbert space.

\paragraph{Lifting to ambient Hilbert space.}
Define the measurable function $\Phi \colon \HH \to \Reals$ by
\begin{equation}
\label{eq:potential_lifted_to_ambient_Hilbert_space}
	\Phi(x)
	\defeq
	\PPhi(\RadProjToSphere(x)), \qquad x \in \HH,
\end{equation}
where $\RadProjToSphere\colon \HH \to \Sphere$ is the radial projection to the sphere from \eqref{eq:radial_projection_map_T}, and define a target measure $\nu \in \probson{\HH}$ via
\begin{equation}
	\label{eq:post}
	\frac{\rd \nu}{\rd \nu_{0}}(x)
	\propto \exp( - \Phi(x)),
	\quad
	\nu_{0}\text{-a.e.\ }x\in\HH,
\end{equation}
where $\nu_{0} = \Normal(0,C)$.
Using $\mu_{0} = \ACG (C) = \RadProjToSphere_{\sharp} \nu_{0}$ and using the construction of $\Phi$, we obtain $\mu = \RadProjToSphere_{\sharp}\nu$.
We show this result in a slightly more general form, i.e.\ for an arbitrary measurable map $T \colon \XX\to\YY$ between two arbitrary topological spaces $\XX$ and $\YY$.
In particular, one can apply \Cref{propo:proj_meas} to more general manifolds in Hilbert spaces, provided that these manifolds can be expressed as the images of a measurable mapping.

\begin{proposition}
	\label{propo:proj_meas}
	Let $\nu_{0} \in \probson{\XX}$, let $T\colon \XX \to \YY$ be measurable, and let $\PPhi \colon \YY \to \Reals$ be such that $Z = \int_{\YY} \exp(- \PPhi(T(x))) \, \nu_{0}(\rd x)<\infty$.
	Define $\nu$ by
	\[
		\frac{\rd \nu}{\rd \nu_{0}}(x)
		=
		\frac{1}{Z} \exp(- \PPhi(T(x))),
		\qquad
		\nu_{0} \text{-a.e.\ }x\in\XX.
	\]
	Then
	\[
		\frac{\rd T_{\sharp}\nu}{\rd T_{\sharp} \nu_{0}}(\xx)
		=
		\frac{1}{Z} \exp(- \PPhi(\xx)),
		\qquad
		T_{\sharp} \nu_{0} \text{-a.e.\ }\xx \in \YY.
	\]
\end{proposition}

\begin{proof}
	Let $A \in \Borel{\YY}$.
	We shall show that
	\begin{equation}
		\label{eq:goal_of_propo_proj_meas}
		(T_{\sharp}\nu)(A)
		= \frac{1}{Z} \int_{A} \exp(-\PPhi(\xx)) \, T_{\sharp} \nu_{0}(\rd \xx).
	\end{equation}
	To this end, let $X\sim \nu_{0}$ and $\widebar{X}\defeq T(X)$, i.e.\ $\widebar{X} \sim T_{\sharp} \nu_{0}$, be random variables on the underlying probability space $(\Omega, \mathcal{A}, \Prob)$ that we fixed in \Cref{sec:preliminaries}.
	Then
	\begin{align*}
		(T_{\sharp}\nu) (A)
		& = \frac{1}{Z} \int_{T^{-1}(A)} \exp(-\PPhi(T(x))) \, \nu_{0}(\rd x) \\
		& = \frac{1}{Z} \int_{X^{-1}(T^{-1}(A))} \exp(-\PPhi(T(X(\omega)))) \, \Prob(\rd \omega) && \text{since $\nu_{0}=\Prob\circ X^{-1}$}\\
		& = \frac{1}{Z} \int_{\widebar{X}^{-1}(A)} \exp(-\PPhi(\widebar{X}(\omega))) \, \Prob(\rd \omega) && \text{since $\widebar{X} \defeq T(X)$}\\
		& = \frac{1}{Z} \int_{A} \exp(-\PPhi(\xx)) \, T_{\sharp} \nu_{0}(\rd \xx) && \text{since $\widebar{X}\sim T_{\sharp} \nu_{0}$,}
	\end{align*}
	which establishes \eqref{eq:goal_of_propo_proj_meas} and completes the proof.
\end{proof}

The idea of sampling the push-forward $\mu = T_{\sharp} \nu$ of a measure $\nu$ defined on the ambient Hilbert space $\HH$ is crucial for the construction of the following algorithms.
In particular, we shall exploit suitable transition kernels for sampling from $\nu \in \probson{\HH}$ in order to construct Markov chains on $\Sphere$ with invariant distribution $T_{\sharp}\nu$, where $T = \RadProjToSphere$.
To this end, we use the framework of push-forward transition kernels, which we describe in \Cref{sec:reprojection_method}.

\subsection{Related approaches and their shortcomings}
\label{ssec:alternative_projection_methods}

Given a Markov chain $(X_n)_{n\in\Naturals}$ with $\nu$-reversible transition kernel $K$, one can also consider $(T(X_n))_{n\in\Naturals}$ as a sequence of random variables on $\YY$.
In our prototypical setting where $\YY=\Sphere$, the stochastic process $(T(X_n))_{n\in\Naturals}$ is simply the projection of the Markov chain $(X_n)_{n\in\Naturals}$ onto $\Sphere$ via $T = \RadProjToSphere$.
Hence, one can think of this as a simple projection approach.
If the law $\Prob \circ (X_n)^{-1}$ of $X_n$ converges to $\nu$ in the total variation norm as $n \to \infty$, then the law $\Prob \circ (T(X_n))^{-1}$ of $T(X_n)$ will also converge to $T_{\sharp}\nu$ in the total variation norm, since
\[
	\norm{ \Prob \circ (T(X_n))^{-1} - T_{\sharp}\nu}_{\TV}
	\leq
	\norm{ \Prob \circ (X_n)^{-1} - \nu }_{\TV}
\]
due to
\[
	\norm{ T_{\sharp}\rho - T_{\sharp}\nu}_{\TV}
	=
	\sup_{A \in \Borel{\Sphere}} \absval{ \rho(T^{-1}(A)) - \nu(T^{-1}(A)) }
	=
	\sup_{A \in \sigma(T)} \absval{ \rho(A) - \nu(A) }
	\leq
	\norm{ \rho - \nu }_{\TV}
\]
with $\Prob \circ (T(X_n))^{-1} = T_{\sharp}\rho$ and $\Prob \circ (X_n)^{-1} = \rho$, where $\sigma(T)$ denotes the Borel $\sigma$-algebra generated by $T$.
However, the sequence $(T(X_n))_{n\in\Naturals}$ fails, in general, to be a Markov chain \citep[e.g.][]{GloverMitro1990}.
In particular, we provide an explicit counterexample in case of $T = \RadProjToSphere$ in \Cref{sec:image_of_MC_is_not_MC}.
More generally, \citet[Theorem~3]{Rosenblatt1966} considers general Markov processes $(X_i)_{i\in I}$, which may be discrete or continuous in time or space, and gives sufficient and necessary conditions on a measurable mapping $T$ such that $(T(X_i))_{i\in I}$ is again a Markov process.

Another related approach can be constructed as follows.
One could simply define the transition kernel
\begin{equation}
	\label{eq:naive}
	\bar{K} (\xx, A) \defeq K (\bar{x}, (\RadProjToSphere)^{-1}(A))
	\qquad
	\forall \xx \in \YY, A \in \Borel{\YY}.
\end{equation}
We shall refer to the transition kernel $\bar{K}$ in \eqref{eq:naive} as the `na\"{i}ve reprojection kernel'.
We call $\bar{K}$ `na\"{i}ve' because it does not perform averaging with respect to the regular conditional distribution $\nu_{\mid T}(\xx,\quark)$ of $X \sim \nu$ given $T(X)=\xx$.
In \eqref{eq:TM_gen} below, we describe a kernel --- the so-called `push-forward transition kernel' --- that does perform this averaging.
In the setting where the topological spaces $\XX$ and $\YY$ satisfy $\YY \subset \XX$, one realises $\yy$ \wrt $\bar{K} (\xx, \quark)$ by first choosing $y$ according to $K (\xx, \quark)$ and then setting $\yy \defeq T(y)$, as illustrated in \Cref{fig:naive-reproj}.
That is, one first transitions from $\xx \in \YY$ to a state $y$ in the ambient space $\XX$, and then ``reprojects'' this state $y$ into $\yy \in \YY$ using the mapping $T$.

\begin{figure}[t]
	\centering
	\includegraphics{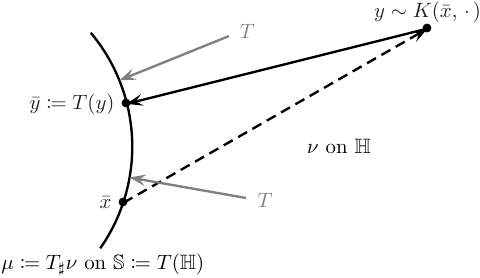}
	\caption{Illustration of the steps for drawing states using the na{\"\i}ve reprojection kernel $\bar{K}$ in \eqref{eq:naive} for $\XX=\HH$, $\YY=\Sphere$, and $T=\RadProjToSphere$ in \eqref{eq:radial_projection_map_T}.
	Starting from $\xx$, an intermediate state $y\in \HH$ is drawn from the $\nu$-reversible transition kernel $K(\xx,\quark)$ on $\HH$.
	The next state drawn from the na\"{i}ve reprojection kernel $\bar{K}(\xx,\quark)$ is then $\yy \defeq T(y)=\RadProjToSphere(y)$.
	Solid arrows indicate deterministic maps, whereas dashed arrows indicate randomised maps, i.e.\ draws from transition kernels.}
	\label{fig:naive-reproj}
\end{figure}
Unlike the projection approach described earlier, this method yields a Markov chain.
However, as numerical experiments show, the na\"{i}ve reprojection kernel $\bar{K}$ does not have $\mu$ as its stationary distribution, even if $K$ is $\nu$-invariant or $\nu$-reversible.
To see why, recall that $\mu=T_\sharp\nu$.
Hence, $\bar{K}$ is $\mu$-invariant if and only if
\begin{align*}
	\int_{\YY} \bar{K} (\xx, A) \, \mu(\rd \xx)
	& =
	\int_{\YY} K (\bar{x}, T^{-1}(A)) \, \mu(\rd \xx)
	=
	\int_{\XX} K (T(x), T^{-1}(A)) \, \nu(\rd x)
	=
	\nu(T^{-1}(A))
\end{align*}
for all $A\in\Borel{\YY}$.
If $K$ is $\nu$-invariant, then by definition $\nu(T^{-1}(A)) = \int_{\XX} K (x, T^{-1}(A)) \, \nu(\rd x)$.
This yields the following necessary and sufficient condition for the $\mu$-invariance of $\bar{K}$:
\begin{equation*}
	\int_{\XX} K (x, T^{-1}(A)) \, \nu(\rd x)
	=
	\int_{\XX} K (T(x), T^{-1}(A)) \, \nu(\rd x),
	\qquad
	\forall A\in\Borel{\YY}.
\end{equation*}
Based upon numerical experiments, we argue that this condition is not necessarily satisfied in the setting where $\XX=\HH$, $\YY=\Sphere$, and $T=\RadProjToSphere$.
Let $\HH = \Reals^3$, $\nu = \Normal(0,C)$ with covariance matrix $C\in\Reals^{3\times 3}$, and consider the $\nu$-reversible \ac{pCN} proposal kernel
\[
	K(x) = \Normal \Bigl( \sqrt{1-s^2}x, s^2 C \Bigr), \quad \text{ with } \quad s=0.7 \text{ and } \
	C = \begin{pmatrix}
	1.25 & 0.33 & -1.62\\
	0.33 & 0.42 & -0.09\\
      -1.62 & -0.09 & 2.85
	\end{pmatrix}.
\]
We now estimate and compare the probability density function of the marginals of $\mu = \RadProjToSphere_\sharp\nu$ and $\mu\bar{K}$ by
kernel density estimation based on $10^6$ independent samples of $\mu$ and $\mu\bar{K}$, respectively.\footnote{The samples were generated as follows: 1) Draw a sample $x$ from $\nu$ and set $\xx \defeq \RadProjToSphere(x)$, so that $\xx$ is a sample draw from $\mu$; 2) draw another sample $w$ from $\nu$ and set $y \defeq \sqrt{1-s^2} \xx + s w$, so that $\yy = \RadProjToSphere(y)$ is a sample draw from $\mu\bar{K}$.}
The results are displayed in \Cref{fig:CounterExam}.
The important observation is that the marginals of $\mu$ (dashed yellow line) and $\mu\bar{K}$ (dotted blue line) differ.
Hence, $\bar{K}$ is not $\mu$-invariant in this case.
Note that the marginals of $\mu$ (dashed yellow line) coincide with the marginals of $\mu(\RadProjToSphere_{\sharp} K)$ (solid red line), where $\RadProjToSphere_{\sharp} K$ is the $\mu$-reversible transition kernel of the reprojected Markov chain using the push-forward transition kernel $T_{\sharp}K$ in \eqref{eq:TM_gen_0}.
\begin{figure}[t]
	\includegraphics[width=0.32\textwidth]{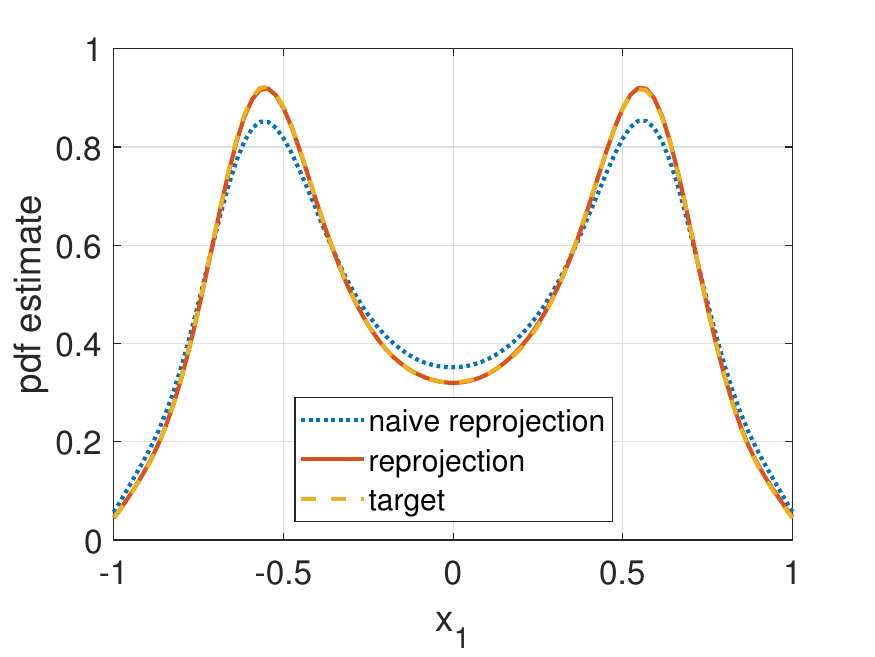}
	\includegraphics[width=0.32\textwidth]{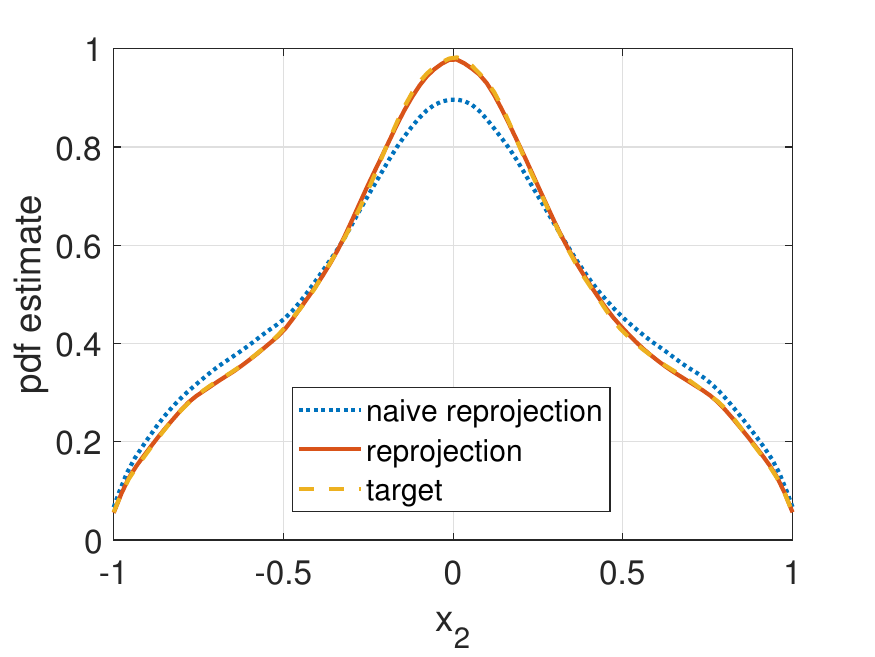}
	\includegraphics[width=0.32\textwidth]{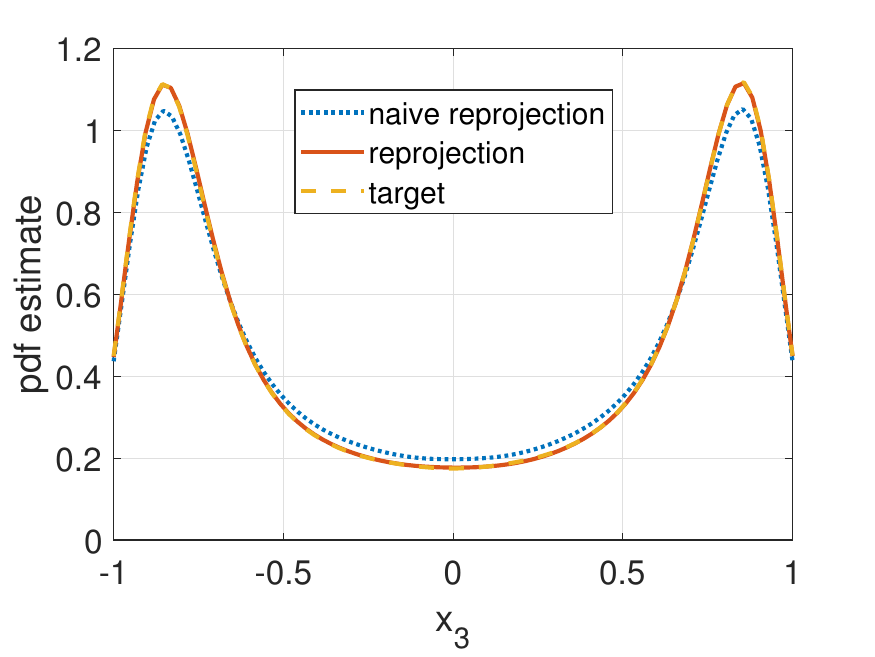}
	\caption{Comparison of marginals of $\mu = \RadProjToSphere_\sharp\nu$ (target), $\mu\bar{K}$ (na\"{i}ve reprojection) with $\bar{K}$ as in \eqref{eq:naive}, and $\mu(\RadProjToSphere_{\sharp} K)$ (reprojection) with $\RadProjToSphere_{\sharp} K$ as in \eqref{eq:TM_gen} with \ac{pCN}-kernel $K$ and $T=\RadProjToSphere$.}
	\label{fig:CounterExam}
\end{figure}

\subsection{The reprojection method}
\label{sec:reprojection_method}
We present now a simple method for defining a $\mu$-reversible Markov chain $(\widebar{X}_n)_{n\in\Naturals}$ on $\Sphere$, by using a $\nu$-reversible transition kernel $K$ on $\HH$.
The method employs the concept of \emph{push-forward transition kernels} which we explain for the general setting of Polish spaces $\XX,\YY$ connected via a measurable mapping $T\colon \XX\to\YY$.
For more details of this approach we refer to \citep{RuSpru2021}.
For the particular algorithms that we consider later, we focus on the specific case of $\YY = \Sphere$, $\XX = \HH$, and $T$ the radial projection map $T=\RadProjToSphere$ given in \eqref{eq:radial_projection_map_T}.

Given a $\nu$-invariant transition kernel $K$ on $\XX$ and a measurable map $T\colon \XX \to \YY$, we define the \emph{push-forward transition kernel} $T_{\sharp}K$ on $\YY$ as follows:
\begin{equation}
\label{eq:TM_gen_0}
	T_{\sharp}K(\xx, A)
	\defeq
	\ev{K(X, T^{-1}(A))\mid T(X) = \xx},
	\qquad
	X\sim\nu,
\end{equation}
where $\xx \in \YY$ and $A \in\Borel{ \YY}$.
If $T$ is bijective, then
\[
	T_{\sharp}K(\xx, A)
	=
	K(T^{-1}(\xx), T^{-1}(A)).
\]
In the following, we also use the shorter notation $\KK = T_{\sharp}K$.
Below, we summarise some important properties of push-forward transition kernels that are inherited from the original transition kernel.

\begin{lemma}[{\citealp{RuSpru2021}}]
	\label{lem:RuSpru}
	Let $\XX$ and $\YY$ be Polish spaces, $T \colon \XX \to \YY$ be a measurable mapping, $K$ be a $\nu$-invariant transition kernel on $\XX$, and $\mu \defeq  T_{\sharp}\nu$.
	\begin{enumerate}[label=(\alph*)]
		\item
		\label{item:RuSpru_!}
		If $K$ is reversible \wrt $\nu$, then $T_{\sharp} K$ is reversible \wrt $\mu$.

		\item
		\label{item:RuSpru_2}
		If $K$ has an $L^2_\nu$-spectral gap, then $T_{\sharp} K$ has an $L^2_\mu$-spectral gap and $$\gap_\mu(T_{\sharp} K) \geq \gap_\nu(K).$$

		\item
		\label{item:RuSpru_3}
		If $K$ is an \ac{MH} kernel with proposal kernel $Q\colon \XX \times \Borel{\XX}\to[0,1]$ and acceptance probability $\alpha \colon \XX \times \XX \to[0,1]$ such that
		\[
			\alpha( x, y ) = \widebar\alpha(T(x),T(y)) \qquad \forall x,y\in\XX
		\]
		for a measurable $\widebar{\alpha} \colon \YY \times \YY \to [0,1]$, then $T_{\sharp} K$ is an \ac{MH} kernel with acceptance probability $\widebar{\alpha}$ and proposal kernel $\widebar Q$ given by
		\[
			\widebar Q(\xx, A)
			\defeq
			\int_{\XX} Q(x , T^{-1}(A)) \, \nu_{|T}(\xx, \rd x),\qquad \forall \xx \in \YY,\ A\in\Borel{\YY},
		\]
		where $\nu_{|T}(\xx,\quark)$ denotes the regular conditional distribution of $X \sim \nu$ given $T(X)=\xx$.
	\end{enumerate}
\end{lemma}

See \Cref{sec:MCMC_general} for the definition of the spectral gap of a transition kernel $K$.

The last item in the above lemma also shows that one can simulate push-forward transition kernels by exploiting the regular conditional distribution $\nu_{\mid T} \colon \YY \times \Borel{\XX} \to [0,1]$ of $X \sim \nu$ given $T(X)=\xx$.
We recall that $\nu_{\mid T}$ possesses the properties of a transition kernel and satisfies
\begin{equation}
	\label{eq:regular_conditional_distribution_nu_given_T}
	\nu_{|T}(T(X), A) = \Prob\left(X \in A\ \middle|\ T(X) \right)
	\qquad
	\Prob\text{-almost surely},
\end{equation}
for any $A \in \Borel{\XX}$.
Given this regular conditional distribution, the disintegration theorem yields the representation
\begin{equation}
\label{eq:TM_gen}
	T_{\sharp}K(\xx, A)
	=
	\int_{\XX}
	K(x , T^{-1}(A)) \, \nu_{|T}(\xx, \rd x),
\end{equation}
for general transition kernels $K$.
Thus, the push-forward transition kernel $T_{\sharp}K$ can be realised by the following mechanism.
\begin{myAlgorithm}
	Given the current state $\xx \in \YY$ one obtains the next state $\yy\in\YY$ as follows:
	\begin{enumerate}[label=(\arabic*)]
		\item Draw $X\sim \nu_{|T}(\xx, \quark)$ and call the realisation $x\in \XX$;
		\item Draw $Y\sim K(x, \quark)$, call the realisation $y \in \XX$ and return $\yy \defeq T(y)\in \YY$.
	\end{enumerate}
\end{myAlgorithm}

We now consider the specific case of $\XX = \HH$ being a Hilbert space, $\YY = \Sphere$ its unit sphere and $T = \RadProjToSphere$ being the radial projection defined in \eqref{eq:radial_projection_map_T}.
In order to obtain a $\mu$-reversible Markov chain on $\Sphere$, we can consider the push-forward transition kernels $\widebar{K} = T_{\sharp}K$ of $\nu$-reversible transition kernels $K$ on the ambient Hilbert space $\HH $ --- such as the \ac{pCN}-\ac{MH} kernel or the \ac{ESS} kernel --- provided that we can also simulate the regular conditional distribution $\nu_{\mid T}$ for the lifted target $\nu$ in \eqref{eq:post}.
The resulting algorithm is illustrated in \Cref{fig:reproj}.
In particular, by going randomly from $\xx$ to $x$ in the ambient space, performing a transition from $x$ to $y$ by using $K$, and then by ``reprojecting'' deterministically from $y$ to $\yy = \RadProjToSphere(y)$, we end up on the sphere $\Sphere$.
Since this is performed at each iteration of the Markov chain, we name this the \emph{reprojection method}.

\begin{figure}
	\centering
	\includegraphics{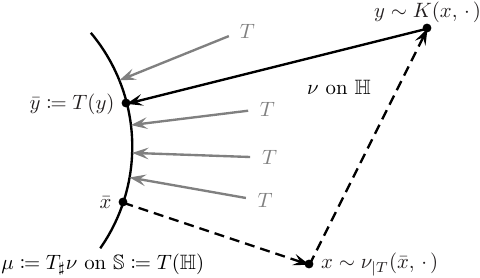}
	\caption{Illustration of the steps in the reprojection method from \Cref{sec:reprojection_method}.
	The reprojection method defines a $\mu$-reversible transition kernel on $\Sphere$ in terms of a $\nu$-reversible transition kernel $K$ on the ambient space $\HH$, where $\Sphere \defeq T(\HH)$, $\mu \defeq T_{\sharp} \nu$, $\nu_{|T} (\xx, \quark)$ is the regular conditional distribution of $X \sim \nu$ given $T(X) = \xx$, and $T=\RadProjToSphere$.
	Solid arrows indicate deterministic maps, whereas dashed arrows indicate randomised maps, i.e.\ draws from transition kernels.}
	\label{fig:reproj}
\end{figure}

Next, we derive the regular conditional distribution $\nu_{|\RadProjToSphere}$ for Gaussian measures $\nu = \Normal(0,C)$ on $\HH = \Reals^d$, and state the resulting reprojected \ac{pCN}-\ac{MH} algorithm as well as the reprojected \ac{ESS} algorithm.

\paragraph{Simulating the conditional distribution $\nu_{|\RadProjToSphere}$.}
We first prove a proposition about the regular conditional distributions $\nu_{|T}$ in the more general setting of Polish spaces $\XX,\YY$.
We then apply this proposition to derive an explicit description for $\nu_{|\RadProjToSphere}$.
One can modify this procedure for manifolds in Hilbert spaces that are more general than the unit sphere $\Sphere$, e.g.\ manifolds that can be described by a measurable mapping $T \colon \HH \to \HH$, by replacing $\RadProjToSphere$ with $T$.

\begin{proposition}
	\label{propo:equality_of_RCDs}
	Let $\XX$ and $\YY$ be Polish spaces equipped with a measurable mapping $T \colon \XX \to \YY$ and $\nu_{0}\in\probson{\XX}$.
	Let $\PPhi \colon \YY \to \Reals$ be measurable with $Z \defeq \int_{\YY} \exp(- \PPhi(T(x))) \, \nu_{0}(\rd x)<\infty$, define $\Phi \colon \XX \to \Reals$ by $\Phi(x) \defeq \PPhi(T(x))$ for $ x \in \XX$, and let $\nu\in\probson{\XX}$ be given by
	\begin{equation*}
	 \frac{\rd \nu}{\rd \nu_{0}}(x)
	 = \frac{1}{Z} \exp( - \Phi(x)),
	\quad
	\nu_{0}\text{-a.e.\ }x\in\XX.
	\end{equation*}
    Furthermore, let $\nu_{0|T}$ be the regular conditional distribution of $X_{0}\sim \nu_{0}$ given $T(X_{0})$, and let $\nu_{|T}$ be the regular conditional distribution of $X\sim \nu$ given $T(X)$.
	Then
	\[
		\nu_{|T}(T(x),\quark) = \nu_{0|T}(T(x),\quark)
	\]
	for $\nu_{0}$-a.e.\ $x\in\XX$.
\end{proposition}

The hypotheses of the proposition differ from the hypotheses of \Cref{propo:proj_meas} only in the additional assumption that $\XX$ and $\YY$ are Polish spaces.
This assumption ensures that we may apply the disintegration theorem to obtain regular conditional distributions.
We can weaken the assumption by requiring that $\XX$ and $\YY$ be merely Radon spaces.

\smallskip

\begin{proof}[Proof of \Cref{propo:equality_of_RCDs}]
	The regular conditional distribution $\nu_{0|T}\colon \YY \times \Borel{\XX} \to [0,1]$ of $X_0$ given $T(X_0)$ is defined as a Markov kernel satisfying for every $A\in\Borel{\XX}$ and $B\in\Borel{\YY}$,
	\[
		\Prob(X_0\in A,\ T(X_0)\in B)
		=
		\int_B \nu_{0|T}(y, A) \, T_{\sharp} \nu_{0}(\rd y)
		=
		\int_{T^{-1}(B)} \nu_{0|T}(T (x), A)\, \nu_{0}(\rd x).
	\]
	Moreover, given the regular conditional distribution  $\nu_{0|T}$ we can express the conditional expectation of $g(X_0)$ given $T(X_0)$ for any measurable $g\colon \XX\to\Reals$ as
	\begin{equation}
		\label{eq:CE_RCD}
		\ev{g(X_0) \mid T(X_0)} = \int_{\XX} g(x) \ \nu_{0|T}(T(X_0), \rd  x)
		\qquad
		\Prob\text{-almost surely}.
	\end{equation}
	We ask now for the regular conditional distribution $\nu_{\vert T}$ of $X\sim \nu$ given $T(X)$.
	Analogously, this is a Markov kernel $\nu_{\vert T}\colon \YY \times \Borel{\XX} \to [0,1]$ satisfying
	\begin{align*}
		\Prob(X\in A,\ T(X)\in B)
		=
		\int_B \nu_{|T}(y, A) \, T_{\sharp}\nu(\rd y)
		=
		\int_{T^{-1}(B)} \nu_{|T}(T (x), A) \, \nu(\rd x)
	\end{align*}
	for every $A\in\Borel{\XX}$ and $B\in\Borel{\YY}$.
	Since $\Prob(X\in A,\ T(X)\in B) = \nu(A\cap T^{-1}(B))$, the statement follows if
	\begin{equation}
		\label{eq:remark_on_RCD_nu_identity}
		\nu(A\cap T^{-1}(B))
		=
		\int_{T^{-1}(B)} \nu_{0|T}(T (x), A)\, \nu(\rd x)
		\qquad
		\forall A\in\Borel{\XX} \  \forall B\in\Borel{\YY},
	\end{equation}
	because $\nu_{0|T}$ is then a valid regular conditional distribution of $X\sim \nu$ given $T(X)$.
	For $A\in\Borel{\XX}$ and $B\in\Borel{\YY}$,
	\begin{align*}
		\nu(A\cap T^{-1}(B)) & = \int_{\XX} \one_A(x) \ \one_B(T(x)) \ \frac 1Z \e^{-\Phi(x)} \ \nu_0(\rd x)
		= \ev{  \frac 1Z \e^{-\Phi(X_0)} \  \one_A(X_0) \  \one_B(T(X_0))},
	\end{align*}
	where $X_0\sim \nu_0$.
	Using the law of total expectation and the hypothesis that $\Phi(x) = \PPhi(T(x))$ for every $x\in \XX$, we obtain
	\begin{align*}
		\nu(A\cap T^{-1}(B)) & = \ev{ \ev{  \frac 1Z \e^{-\Phi(X_0)} \  \one_A(X_0) \  \one_B(T(X_0))
				\,\Big|\, T(X_0)} }\\
		& = \ev{ \frac 1Z \e^{-\Phi(X_0)} \ \one_B(T(X_0)) \ \ev{  \one_A(X_0) \mid T(X_0)} }.
	\end{align*}
	Applying now \eqref{eq:CE_RCD} to $g(x) = \one_A(x)$ yields $\ev{ \one_A(X_0) \mid T(X_0)} = \nu_{0|T}(T(X_0), A)$ and, thus,
	\begin{align*}
		\nu(A\cap T^{-1}(B)) & = \ev{ \frac 1Z \e^{-\Phi(X_0)} \ \one_B(T(X_0)) \ \nu_{0|T}(T(X_0), A) }\\
		& = 	\int_{T^{-1}(B)} \frac 1Z \e^{-\Phi(x)} \ \nu_{0|T}(T(x), A) \ \nu_0(\rd x)\\
		& =  \int_{T^{-1}(B) } \nu_{0|T}(T(X_0), A) \ \nu(\rd x)
	\end{align*}
	which shows \eqref{eq:remark_on_RCD_nu_identity}.
\end{proof}

We now turn to the setting of a finite dimensional sphere, where $\XX=\Reals^d$, $\YY=\Sphere^{d-1}$, and $T=\RadProjToSphere$.
In order to implement the reprojection method, it suffices by \Cref{propo:equality_of_RCDs} to simulate the regular conditional distribution of $X_{0}\sim \Normal(0,C)$ given $\RadProjToSphere(X_{0}) = \xx$.
In the following result $\Gam(a,b)$ denotes the Gamma distribution with shape parameter $a>0$ and inverse scale parameter $b>0$.

\begin{proposition}
	\label{propo:nu_T}
	Let $\nu_0 = \Normal(0,C)$ be given on $\Reals^d$ and let $\nu_{0 \mid \RadProjToSphere}(\xx, \quark)$ denote the conditional distribution of $X_{0} \sim \nu_{0}$ given $\RadProjToSphere(X_{0}) = \xx$.
	Then, for a non-negative real-valued random variable $R$ satisfying $R^2 \sim \Gam\left(\frac{d}{2}, \frac{1}{2} \xx^\top C^{-1} \xx\right)$,
	\[
		R \xx \sim \nu_{0 \mid \RadProjToSphere}(\xx, \quark).
	\]
\end{proposition}

\begin{proof}
We can write $X_{0}\sim \nu=\Normal (0,C)$ as $X_{0} = R\widebar{X}_{0}$, $\widebar{X}_{0} \defeq \RadProjToSphere(X_{0})$, $R \defeq \norm{ X_{0} }$.
Thus, the condition of $\RadProjToSphere(X_{0}) = \xx$ yields the following conditional density of $R$:
\[
	f_{R|\RadProjToSphere(X_{0}) = \xx}(r)
	\ \propto \
	r^{d-1} \exp\left(-\frac{1}{2} ( \xx^\top C^{-1} \xx) r^2 \right)
	=
	r^{d-1} \left( \exp\left(-\frac{1}{2} \xx^\top C^{-1} \xx\right)\right)^{r^2}.
\]
By the change of variables $r \mapsto r^2 \qefed r_2$ we obtain the following probability density for $R^2$ given $\RadProjToSphere(X_{0}) = \xx$:
\[
	f_{R^2|\RadProjToSphere(X_{0}) = \xx}(r_2)
	\ \propto \
	\frac{r_2^{(d-1)/2}}{2r_2^{1/2}} \left( \e^{-\frac{1}{2} \xx^\top C^{-1} \xx}\right)^{r_2}
	\ \propto \
	r_2^{d/2-1} \e^{- \left(\frac{1}{2} \xx^\top C^{-1}\xx \right) {r_2}}.
\]
Thus, $R^2$ conditioned on $\RadProjToSphere(X_{0})=\xx$ follows the $\Gam(\frac{d}{2}, \tfrac{1}{2} \xx^\top C^{-1} \xx)$, as desired.
\end{proof}

\paragraph{Resulting algorithms}
We now provide two explicit algorithms for approximate sampling of target measures $\mu$ on $\Sphere^{d-1}$ as given in \eqref{eq:push_post}.
\Cref{alg:pCN_Sphere,alg:ESS_Sphere} result from applying the push-forward transition kernel approach to the \ac{pCN}-\ac{MH} algorithm and the \ac{ESS} algorithm on $\HH$, respectively.
The \ac{pCN}-\ac{MH} algorithm and the \ac{ESS} algorithm on $\HH$ were stated in \Cref{alg:pCN_on_H,alg:ell_slice_sampling}.
\begin{algorithm}
  \caption{Reprojected \ac{pCN}-\ac{MH} algorithm on $\Sphere^{d-1}$} \label{alg:pCN_Sphere}
 \begin{algorithmic}[1]
  \STATE \textbf{Given:} ACG prior $\mu_0 = \ACG (C)$ and target $\mu$ as in \eqref{eq:push_post}
  \STATE \textbf{Initial:} step size $s\in(0,1]$ and state $\xx_0 \in \Sphere^{d-1}$
  \FOR{$k \in\Naturals_0$}
	\STATE Draw a sample $r_k^2$ of $ \Gam\left(d/2, \frac{1}{2} \xx_k^\top C^{-1} \xx_k\right)$ and set $x_k = r_k \xx_k$
	\STATE Draw a sample $w_k$ of $\Normal(0, C)$ and set $y_{k+1} \defeq  \sqrt{1-s^2} x_k + s w_k  $
	\STATE Set $\yy_{k+1} \defeq  y_{k+1} / \norm{ y_{k+1} }$
	\STATE Compute $a \defeq  \min\{1, \exp\left(\bar\Phi(\xx_{k}) - \bar\Phi(\yy_{k+1}) \right)\}$
	\STATE Draw a sample $u$ of $\Uniform[0,1]$
	\IF{$u\leq a$}
		\STATE Set $\xx_{k+1} = \yy_{k+1}$
	\ELSE
		\STATE Set $\xx_{k+1} = \xx_{k}$
	\ENDIF
  \ENDFOR
    \end{algorithmic}
\end{algorithm}
\begin{algorithm}
  \caption{Reprojected \ac{ESS} algorithm on $\Sphere^{d-1}$} \label{alg:ESS_Sphere}
 \begin{algorithmic}[1]
  \STATE \textbf{Given:} ACG prior $\mu_0 = \ACG (C)$ and target $\mu$ as in \eqref{eq:push_post}
  \STATE \textbf{Initial:} state $\xx_0 \in \Sphere^{d-1}$
  \FOR{$k \in\Naturals_0$}
	\STATE Draw a sample $t\sim \Uniform[0,\exp(-\bar\Phi(\xx_k))]$
	\STATE Draw a sample $r_k^2$ of $ \Gam\left(d/2, \frac{1}{2} \xx_k^\top C^{-1} \xx_k\right)$ and set $x_k = r_k \xx_k$
	\STATE Set $x_{k+1}=\text{shrink-ellipse}(x_k,t)$ \; (see \Cref{alg:ell_shrink})
	\STATE Set $\xx_{k+1}=x_{k+1}/\norm{ x_{k+1} }$
\ENDFOR
    \end{algorithmic}
\end{algorithm}
According to \Cref{lem:RuSpru}\ref{item:RuSpru_3}, \Cref{alg:pCN_Sphere} yields an \ac{MH} algorithm on $\Sphere^{d-1}$.
Its acceptance probability is simply $\widebar\alpha(\xx,\yy) =  \min\{1, \exp\left(\bar\Phi(\yy )- \bar\Phi(\xx) \right)\}$, for $\xx,\yy \in \Sphere^{d-1}$, and its proposal kernel $\widebar Q\colon \Sphere^{d-1} \times \Borel{\Sphere^{d-1}} \to [0, 1]$ admits a proposal density $\widebar q\colon \Sphere^{d-1} \times \Sphere^{d-1} \to (0,\infty)$ \wrt the Hausdorff measure on $\Sphere^{d-1}$ given by
\begin{equation}\label{eq:qbar}
	\widebar q(\xx,\yy)
	=
	\int_0^\infty
	\int_0^\infty
	q(r\xx, r' \yy)
	\
	f_{\xx}(r) \, (r')^{d-1} \, \rd r \, \rd r' > 0,
	\qquad
	\xx,\yy \in \Sphere^{d-1},
\end{equation}
where $q$ denotes the proposal density of the \ac{pCN} proposal kernel $Q(x,\quark) = \Normal(\sqrt{1-s^2}x,s^2 C)$, $x\in\Reals^d$, and $f_{\xx}(r)$ denotes the conditional density of $R = \norm{ X }$ for $X\sim \Normal(0,C)$ given $\RadProjToSphere(X) = \xx$.
According to the proof of \Cref{propo:nu_T}, the density $f_{\xx}(r)$ takes the form
\begin{equation}\label{eq:fxx}
	f_{\xx}(r)
	\defeq
	\frac 1{c_{\xx}} r^{d-1} \exp\left(-\frac{r^2}{2} \xx^\top C^{-1} \xx\right),
	\qquad
	c_{\xx} \defeq  \int_0^\infty r^{d-1} \exp\left(-\frac{r^2}{2} \xx^\top C^{-1} \xx\right) \rd r.
\end{equation}

\begin{remark}
Due to the generality of the pushforward Markov kernel approach, it is also possible to combine the reprojection methodology that we proposed with other common MH algorithms, such as the dimension-independent Hamiltonian Monte Carlo (HMC) algorithm of \cite{Beskos_etal_2011}.
However, an extensive investigation of a reprojected HMC algorithm that shows advantages and disadvantages, e.g.\ in comparison to the work of \citet{LanZhouShahbaba2014}, is beyond the scope of this paper.
\end{remark}

\subsection{Uniform and geometric ergodicity}

We investigate the exponential convergence behaviour of the transition kernels that correspond to the Markov chains that are realised either by \Cref{alg:pCN_Sphere} or \Cref{alg:ESS_Sphere}.
Since the underlying state space $\Sphere^{d-1}$ is compact, we aim for uniform ergodicity.
The associated transition kernel $\widebar{K}\colon \Sphere^{d-1} \times \Borel{\Sphere^{d-1}} \to [0, 1]$ is said to be \emph{uniformly
	ergodic}, if there are $\kappa \in [0,1)$ and $c<\infty$ such that
\begin{equation}\label{eq:uniform_ergodic}
	\norm{ \widebar{K}^{n}(\xx) - \mu }_{\TV}
	\leq c \, \kappa^{n}
	\qquad
	\forall \xx \in \Sphere^{d-1}.
\end{equation}
It is well known
\citep[e.g.][Theorem~16.0.2]{MeTw09}
that uniform ergodicity of a Markov chain is equivalent to the \emph{smallness} of the whole state space.
A set $B \in\Borel{ \Sphere^{d-1}}$ is called \emph{small} \wrt a transition kernel $\widebar{K}$ if there exists some $m\in\Naturals$ and a nonzero measure $\phi$ on $(\Sphere^{d-1},\Borel{ \Sphere^{d-1}})$ such that
\begin{equation}\label{eq:small}
	\widebar{K}^m(\xx, A) \geq \phi(A) \qquad \forall A \in \Borel{\Sphere^{d-1}},\ \xx \in B.
\end{equation}
In particular, if \eqref{eq:small} holds for $B=\Sphere^{d-1}$, then \citep[Theorem~16.2.4]{MeTw09} yields that
\begin{equation}
\label{eq:smallness_unif_erg}
	\norm{ \widebar{K}^{n}(\xx) - \mu }_{\TV}
	\leq (1-\phi(\Sphere^{d-1}))^{n/m-1}
\end{equation}
By exploiting the particular structure \eqref{eq:TM_gen} of push-forward transition kernels, we obtain the following result.

\begin{theorem}
\label{theo:Uniform_ergodic}
Let $\PPhi\colon \Sphere^{d-1} \to \Reals$ be uniformly bounded.
Then the transition kernels corresponding to the Markov chains realised by \Cref{alg:pCN_Sphere,alg:ESS_Sphere} are uniformly ergodic.
\end{theorem}
\begin{proof}
The idea of the proof is to show that, in both cases, the state space is small.

We first consider the reprojected \ac{pCN}-\ac{MH} kernel $\widebar{K}$.
The boundedness of $\PPhi$, i.e.\ $\underline{c} \leq \PPhi(\xx) \leq \overline{c}$, yields the following lower bound on the acceptance probability:
\[
	\widebar{\alpha}(\xx,\yy) = \min\left\{1, \exp\left(\PPhi(\yy) - \PPhi(\xx) \right) \right\}
	\geq
	\exp(\underline{c} - \overline{c}) > 0.
\]
Hence, by the corresponding \ac{MH} form of the reprojected \ac{pCN}-\ac{MH} kernel $\widebar{K}$ stated in \Cref{lem:RuSpru}, for any $A \in\Borel{\Sphere^{d-1}}$ and any $\xx \in \Sphere^{d-1}$,
\[
	\widebar{K}(\xx, A)
	\geq
	\int_A
	\widebar{\alpha}(\xx,\yy)
	\
	\widebar Q(\xx, \rd \yy)
	\geq
	\exp(\underline{c} - \overline{c})
	\
	\widebar Q(\xx, A).
\]
Recall that $\widebar Q$ possesses the density $\widebar q$ given in \eqref{eq:qbar}.
Note that $\xx \mapsto \xx^\top C^{-1}\xx$ is bounded on $\Sphere^{d-1}$, such that there exists for every $r>0$ a lower bound $\underline{f}(r)>0$ satisfying $f_{\xx}(r) \geq \underline{f}(r) >0$ for every $\xx\in \Sphere^{d-1}$.
Moreover, the density $q(\quark, y)$ of the \ac{pCN} proposal kernel $Q(x,\quark) = \Normal(\sqrt{1-s^2}x,s^2C)$ is continuous.
Therefore, $q(x,y)$ is uniformly bounded away from zero for any $x,y \in \Reals^{d}$ with $\norm{ x },\norm{ y } \leq 1$.
Hence, there exists some $\epsilon>0$ such that, for every $\xx,\yy \in \Sphere^{d-1}$, $\widebar q(\xx,\yy)$ in \eqref{eq:qbar} satisfies $\widebar q(\xx,\yy)\geq \epsilon$.
This implies that $\widebar Q(\xx, A) \geq \epsilon \Hausdorff_{\Sphere^{d-1}}(A)$, for the Hausdorff measure $\Hausdorff_{\Sphere^{d-1}}$ on $\Sphere^{d-1}$.
Thus, $\Sphere^{d-1}$ is small \wrt $\widebar{K}$ with $\phi(A) \defeq  \epsilon \exp(\underline{c} - \overline{c})  \Hausdorff_{\Sphere^{d-1}}(A)$.

Next, we consider the reprojected \ac{ESS} kernel.
Since $\Phi$ on $\HH=\Reals^{d}$ is constructed from $\PPhi$ on $\Sphere$ by \eqref{eq:potential_lifted_to_ambient_Hilbert_space}, the boundedness of $\PPhi$ implies the boundedness of $\Phi$.
Hence, any compact set $B \subset \Reals^{d}$ is small \wrt the \ac{ESS} transition kernel $K$.
The measure \wrt which the smallness property holds is $\phi = \epsilon_B \lambda_B$, where $\epsilon_B>0$ denotes a constant and $\lambda_B$ is the Lebesgue measure restricted to a compact set $B$ with positive $d$-dimensional Lebesgue measure; see \cite[Lemma~3.4]{NataElAl2021}.
We now use this fact in order to show the smallness of $\Sphere^{d-1}$ \wrt the reprojected \ac{ESS} transition kernel $\widebar{K}$ for the measure $\phi = \epsilon \Hausdorff_{\Sphere^{d-1}}$ for appropriately chosen $\epsilon>0$, see below.
To this end, we apply the representation \eqref{eq:TM_gen} with $\XX=\Reals^{d}$, $\YY=\Sphere^{d-1}$ and $T = \RadProjToSphere$:
\[
	\widebar{K}(\xx, A)
	=
	\int_{0}^\infty
	K(r\xx , T^{-1}(A)) f_{\xx}(r) \, \rd r
\]
with $f_{\xx}$ as in \eqref{eq:fxx}.
Again, since $\xx \mapsto \xx^\top C^{-1}\xx$ is bounded on $\Sphere^{d-1}$, there exists for every $r>0$ a lower bound $\underline{f}(r)>0$ such that, for every $\xx\in \Sphere^{d-1}$, $f_{\xx}(r) \geq \underline{f}(r) >0$.
Now fix $B \defeq  B_1(0) = \set{ x\in\Reals^d }{ \norm{ x }\leq 1 }$ and note that $B$ is small \wrt $K$.
Thus
\begin{align*}
	\widebar{K}(\xx, A)
	& \geq
	\int_{0}^{1}
	K(r\xx , T^{-1}(A)) \underline{f}(r) \, \rd r
	\geq
	\epsilon_B \lambda_B(T^{-1}(A)) \, \int_{0}^{1}  \underline{f}(r) \, \rd r
	= \epsilon \Hausdorff_{\Sphere^{d-1}}(A)
\end{align*}
where $\epsilon \defeq  \epsilon_B \int_{0}^{1}  \underline{f}(r) \, \rd r \, \int_{0}^{1}  u^{d-1} \, \rd u$, since $\lambda_B(T^{-1}(A)) = \int_A \int_{0}^{1} u^{d-1} \rd u \, \Hausdorff_{\Sphere^{d-1}}(\rd \xx)$.
\end{proof}

The boundedness assumption on $\PPhi$ in \Cref{theo:Uniform_ergodic} is rather mild.
It is satisfied if $\PPhi\colon \Sphere^{d-1} \to \Reals$ is continuous.
For example, in the Bayesian level set inversion and Bayesian density estimation problems considered in \Cref{sec:numerical_illustrations}, the corresponding $\PPhi$ is bounded.

\Cref{theo:Uniform_ergodic} yields uniform ergodicity in finite dimension.
In the last paragraph of the proof of \Cref{theo:Uniform_ergodic}, we considered the measure $\phi \defeq  \epsilon\exp(\underline{c} - \overline{c})   \Hausdorff_{\Sphere^{d-1}}$ for the reprojected \ac{pCN}-\ac{MH} kernel $\widebar{K}$ and the measure $\phi\defeq \epsilon \Hausdorff_{\Sphere^{d-1}}$ for the reprojected \ac{ESS} kernel.
Supposing that the pre-factors $\epsilon$ and $\epsilon\exp(\underline{c} - \overline{c})$ do not grow in $d$ and substituting these choices of $\phi$ in \eqref{eq:smallness_unif_erg} we observe that the corresponding $\kappa=(1-\phi(\Sphere^{d-1}))^{1/ m}$ in \eqref{eq:uniform_ergodic}, with $c=(1-\phi(\Sphere^{d-1}))^{-1}$, converges exponentially quickly to 1 as $d\to\infty $.
This is because the $(d-1)$-dimensional Hausdorff measure of $\Sphere^{d-1}$ is given by $\tfrac{2\pi^{d/2}}{\Gamma(d/2)}$, where $\Gamma(\quark)$ is the Gamma function, and because of the asymptotic behaviour of $\Gamma(\quark)$.

In the subsequent paragraph, we present a dimension-independent convergence behaviour, but in the context of geometric ergodicity as in \eqref{eq:geometric_ergodic}, and not in the context of uniform ergodicity as in \eqref{eq:uniform_ergodic}.

\paragraph{Dimension-independent geometric ergodicity.}
In order to study the geometric ergodicity of Markov chains generated by the reprojected \ac{pCN}-\ac{MH} and reprojected \ac{ESS} algorithms, we can exploit \Cref{lem:RuSpru}.
This lemma states that the spectral gaps of the reprojected transition kernels $\KK$ of \Cref{alg:pCN_Sphere,alg:ESS_Sphere} are at least as large as the spectral gaps of the transition kernels $K$ of \Cref{alg:pCN_on_H,alg:ell_slice_sampling} respectively.
In order to describe a dimension-independent spectral gap, we introduce the following notation: Given $\mu_0 = \ACG(C)$ with non-degenerate, trace-class covariance operator $C\colon \HH\to\HH$ on an infinite-dimensional separable Hilbert space $\HH$, let $\set{ e_j }{ j\in\Naturals }$ be a complete orthonormal system in $\HH$ consisting of the eigenvectors of $C$.
We now construct finite-dimensional approximations to the infinite-dimensional setting as follows: For $d\in\Naturals$, let $\mu^{(d)}_{0} = \ACG(C_d)$ denote the ACG measure on $\Sphere^{d-1}$ resulting from the marginal of $\Normal(0,C)$ on $\Span{e_1,\ldots,e_d}$ and consider the target measure $\mu^{(d)}$ on $\Sphere^{d-1}$ given by
\begin{equation}
	\label{eq:push_post_finite}
	\frac{\rd \mu^{(d)}}{\rd \mu^{(d)}_{0}}(\xx)
	\propto \exp( - \PPhi(\xx)),
	\quad
	\xx \in \Sphere^{d-1}.
\end{equation}
In order to apply $\PPhi$ to $\xx \in \Sphere^{d-1}$, we view $\Sphere^{d-1}$ as the ``equatorial'' subsphere $\{\xx_1 e_1 + \ldots + \xx_d e_d \in \HH \colon \xx = (\xx_1,\ldots,\xx_d) \in \Sphere^{d-1} \}$ of $\Sphere\subset\HH$.
Let
\begin{equation}
	\label{eq:post_finite}
	\frac{\rd \nu^{(d)}}{\rd \nu^{(d)}_{0}}(x)
	\propto \exp( - \Phi(x)),
	\quad
	x \in \Reals^{d},
\end{equation}
where $\nu^{(d)}_{0} = \Normal(0,C_d)$ and $\Phi(x) \defeq \PPhi(\RadProjToSphere(x))$ for $x\in\HH$, as in \eqref{eq:potential_lifted_to_ambient_Hilbert_space}.
In order to apply $\Phi$ to $x \in \Reals^{d}$, we view $\Reals^{d}$ as the subspace $\{x_1 e_1 + \ldots + x_d e_d \in \HH \colon x = (x_1,\ldots,x_d) \in \Reals^{d} \}$ of $\HH$.

For a reminder of the definition of $\gap_{\mu}(K)$ for a given measure $\mu$ and transition kernel $K$ we refer to \eqref{eq:gap}.
By \Cref{lem:RuSpru} we obtain the following result.

\begin{proposition}\label{propo:Comp}
	Let $\mu^{(d)}$ and $\nu^{(d)}$ be as in \eqref{eq:push_post_finite} and \eqref{eq:post_finite} respectively.
	Let $K^{(d)}$ denote the \ac{pCN}-\ac{MH} transition kernel targeting $\nu^{(d)}$ using a step size $s\in(0,1]$ in the proposal.
	If there exists a $\beta > 0$ such that
	\begin{equation}\label{eq:pCN_gap}
		\inf_{d\in\Naturals} \gap_{\nu^{(d)}}\left( K^{(d)} \right) \geq \beta
	\end{equation}
	then the reprojected \ac{pCN}-\ac{MH} transition kernel $\widebar{K}^{(d)} \defeq (\RadProjToSphere)_\sharp K^{(d)}$ targeting $\mu^{(d)}$ on $\Sphere^{d-1}$ satisfies
	\begin{equation}\label{eq:repro_pCN_gap}
		\inf_{d\in\Naturals} \gap_{\mu^{(d)}}\left( \widebar{K}^{(d)} \right) \geq \beta.
	\end{equation}
    The same statement holds for the reprojected \ac{ESS} transition kernel $\widebar{K}^{(d)} \defeq (\RadProjToSphere)_\sharp K^{(d)}$.
\end{proposition}

Dimension independence of the spectral gap of the \ac{ESS} transition kernel has been demonstrated in the literature by numerical experiments \citep{NataElAl2021}.
However, to the best of our knowledge, no theoretical proof is available.
Therefore, we focus on the \ac{pCN}-\ac{MH} algorithm, for which \eqref{eq:pCN_gap} was shown by \citet{Hairer2014} under certain assumptions on $\Phi$.
For convenience, we summarise their result:

\begin{theorem}
\label{theo:vollmer}
Let $K^{(d)}$ denote the \ac{pCN}-\ac{MH} transition kernel targeting $\nu^{(d)}$ using the step size $s\in(0,1]$ in the proposal.
Suppose the following conditions hold:
\begin{enumerate}[label=(\alph*)]
\item \label{it:Hai_1}
There exist some $R>0$ and $\underline{\alpha} \in \Reals$ such that, for all $x\in\HH$ with $\norm{ x } > R$,
\[
	\Phi(y) < \Phi(x) - \underline{\alpha} \qquad
	\text{for all $y\in\HH$ with $\bignorm{ y - \sqrt{1-s^2}x } \leq \frac12 \bigl( 1-\sqrt{1-s^2} \bigr) \norm{ x }$.}
\]

\item \label{it:Hai_2}
The function $\e^{-\Phi}$ is integrable \wrt $\nu_0 = \Normal(0,C)$.

\item \label{it:Hai_3}
For every $\gamma>0$ there exists some $C_\gamma<\infty$ such that
\[
	|\Phi(x) - \Phi(y)| \leq C_\gamma \e^{\gamma r} \qquad \text{for all $x,y\in\HH$ with $\norm{ x }, \norm{ y } \leq r$.}
\]
\end{enumerate}
Then there exists a $\beta >0$ such that \eqref{eq:pCN_gap} holds.
\end{theorem}

The conditions of \Cref{theo:vollmer} are satisfied for constant $\Phi$.
Thus, the \ac{pCN}-\ac{MH} transition kernel exhibits a dimension-independent spectral gap when targeting the prior $\nu=\nu_0 =  \Normal(0,C)$.
By means of results by \cite{Vollmer2015} and \cite{RuSpru2018}, this can then be lifted to bounded perturbations of the prior measure, such as $\mu$ as in \eqref{eq:potential_lifted_to_ambient_Hilbert_space} for bounded $\PPhi$.
This yields our final result.

\begin{theorem}
\label{theo:main}
Let $\PPhi\colon \Sphere^{d-1} \to \Reals$ be uniformly bounded.
Then the transition kernel corresponding to the Markov chain realised by \Cref{alg:pCN_Sphere} has a dimension-independent spectral gap in the sense of \eqref{eq:repro_pCN_gap}.
\end{theorem}
\begin{proof}
Let $K^{(d)}$ denote the \ac{pCN}-\ac{MH} transition kernel in $\Reals^d$ for an arbitrary step size $s\in(0,1]$ and dimension $d\in\mathbb N$ with target $\nu^{(d)}$, and let 
$K_0^{(d)}$ denote the \ac{pCN}-\ac{MH} transition kernel that targets the prior $\nu_0^{(d)}$.
Note that $K_0^{(d)}$ coincides with the corresponding proposal kernel.
By \Cref{theo:vollmer} we know that there exists a $\beta>0$ such that
\[
		\inf_{d\in\Naturals} \gap_{\nu_0^{(d)}}\left( K_0^{(d)} \right) \geq \beta.
\]
By \cite[Theorem~11]{RuSpru2018} the transition operator $K^{(d)}$ is positive, and hence,
\[
	\gap_{\nu_0^{(d)}}\left( K_0^{(d)} \right)
	=
	1 - \Lambda_{\nu_0^{(d)}}\left(  K_0^{(d)} \right),
	\qquad
	\gap_{\nu^{(d)}}\left( K^{(d)} \right)
	=
	1 - \Lambda_{\nu^{(d)}}\left(  K^{(d)} \right),
\]
where $\Lambda_\mu(K)$ denotes the supremum of the spectrum of the restriction of a $\mu$-invariant transition operator $K$ to $L^2_{0,\mu}$, where
$L^2_{0,\mu}\defeq \set{f\in L^2(\mu)}{\int f(u)\mu(\rd u)=0}$, and $L^2(\mu)\defeq \set{f:\HH\to\Reals}{\norm{f}_{L^2_\mu}\defeq \int \absval{ f }^2\mu(\rd u)<\infty}$.
In the reversible case, $\Lambda_\mu(K) = \sup_{f \in L^2_{0,\mu}} \innerprod{ Kf}{f}_{L^2_{\mu}} / \norm{ f }_{L^2_{\mu}}$.
Now, if $\PPhi$ is bounded, then so is $\Phi$.
A comparison result \cite[Theorem~3.3]{Vollmer2015} states that
\[
	1 - \Lambda_{\nu^{(d)}}\left(  K^{(d)} \right)
	\geq
	\exp\left( 4 \inf \PPhi - 4 \sup \PPhi \right)
	\left( 1 - \Lambda_{\nu_0^{(d)}}\left(  K_0^{(d)} \right) \right).
\]
Applying this comparison result yields
\[
		\inf_{d\in\Naturals} \gap_{\nu^{(d)}}\left( K^{(d)} \right) \geq \exp\left(4 \inf \PPhi - 4 \sup \PPhi\right) \beta \ > \ 0
\]
which by \Cref{propo:Comp} yields the statement.
\end{proof}

\section{Numerical illustrations}
\label{sec:numerical_illustrations}

We now demonstrate the dimension-independent performance of the reprojected \ac{pCN}-\ac{MH} algorithm and the reprojected \ac{ESS} algorithm on two applications. 
\Cref{ssec:level_set_inversion}, rooted in inverse problems, treats an application to Bayesian binary classification or level set inversion. 
\Cref{ssec:numerical_density_estimation} has a more statistical flavour and  considers an application to Bayesian density estimation.
Readers whose main interest is in the second application may proceed directly to \Cref{ssec:numerical_density_estimation} but may wish to recall the definition \eqref{eq:RMSJD} of the root mean square jump distance with respect to the Riemannian metric on the sphere.
We will use these applications to illustrate the dimension-dependent performance of the geodesic random walk-\ac{MH} algorithm of \citet{MangoubiSmith2018} and the \ac{MH} algorithm of \citet{Zappa2018}.

\subsection{Bayesian binary level set inversion}
\label{ssec:level_set_inversion}
For the convenience of the reader, we give a self-contained description of (Bayesian) binary level set inversion, following the presentation of \citet{IglesiasLuStuart2016}.
Readers who are interested in technical details concerning the random fields perspective of level set inversion may consult Appendix 2 of that paper. For simplicity, we focus on the single-phase Darcy flow model or groundwater flow problem on a bounded domain $D\subset \Reals^k$, $k=1,2,3$, with closure $\overline{D}$ and boundary $\partial D$.
This problem is described by the elliptic \ac{PDE}
\begin{subequations}
\label{eq:PDE}
\begin{align}
	-\nabla \cdot (\e^{u} \nabla p) & = f\text{ in }D, \\
	p & = \kappa\text{ on }\partial D,
\end{align}
\end{subequations}
where $p$ denotes a fluid pressure field, $f$ describes sources and sinks, and $u$ is the log-permeability parameter.
Let $\U \defeq L^{\infty}(D)$ and $\V\defeq \set{ \tilde{p}\in H^{1}(D) }{ \tilde{p}=\kappa \text{ on } \partial D }$, where $H^{1}(D)$ is the Sobolev space of functions in $L^2(D)$ whose first-order weak derivatives have finite $L^2(D)$ norm.
If $u\in \U$, then $\e^{u}\in L^\infty(D)$, and given a suitable boundary condition $\kappa$ and source term $f \in L^2(D)$, a unique weak solution $p \in \V$ of \eqref{eq:PDE} exists.
Denote the solution map that maps the log-permeability $u$ to the corresponding solution $p$ by $\Upsilon \colon \U \to \V$.
Then $\Upsilon$ is locally Lipschitz continuous, e.g., \citep{BonitoEtAl2017}.

We assume now that the domain $D$ is divided into two disjoint regions $D_0, D_1\subset D$, i.e.\ $\overline{D} = D_0\cup D_1$.
The subdomains $D_i$ describe the location of different materials, e.g., background and abnormal material, with different constant log-permeabilities $u_0, u_1 \in \Reals$.
In particular, in binary level set inversion we assume that $u$ takes the form
\begin{equation}\label{eq:BLI_u}
	u(t) = u_0 \one_{D_0}(t) + u_1 \one_{D_1}(t), \qquad t \in D,
\end{equation}
where $\one_{D_i}$ denotes the indicator function of $D_i$ and where the values $u_0,u_1$ are known a priori.

The goal is then to infer the location of $D_1$ and, hence, $D_0 = \overline{D}\setminus D_1$ based on noisy observations of $p$, i.e.,
\begin{equation}\label{eq:data}
	y=O\circ \Upsilon (u)+\eta,
\end{equation}
where $O \colon \V \to \Reals^{J}$ denotes for some $J\in\Naturals$ a bounded linear observation operator and $\eta$ describes observational noise.
We assume that $\eta \sim \Normal(0,\Sigma)$ with known covariance $\Sigma\in\Reals^{J\times J}$.

In the level set approach we then introduce a so-called \emph{level set function} $g\in \Sphere(L^2(D)) \cap C(\overline{D})$ as well as the \emph{level set map} $\mathcal{L} \colon C(\overline{D}) \to \U$ defined by
\begin{equation}\label{eq:requirement_on_level_set_function}
	g\mapsto \mathcal{L}  (g)\defeq u_0 \one_{D_0(g)} + u_1 \one_{D_1(g)},
	\qquad
	D_1(g) \defeq \set{ t\in D }{ g(t) \geq 0},
\end{equation}
and $D_0(g)  = \set{ t\in D }{ g(t) < 0}$, respectively.
We can formulate the level set inverse problem as the problem of inferring the data-generating level set function $g$, given the model that the observations $y=O\circ \Upsilon \circ \mathcal{L}(g) + \eta$ are generated by a unique $g \in \Sphere(L^2(D)) \cap C(\overline{D})$.
We argue that the unit sphere in the function space $L^2(D)$ or $C(\overline D)$ is an advantageous setting for the level set inverse problem:
if we considered the problem in the ambient space $g \in C(\overline{D})$, then $g$ becomes non-identifiable, because $\mathcal{L}(\alpha g) =  \mathcal{L}(g)$ for any $\alpha >0$. 
For computational convenience, below we choose to work with the unit sphere in the Hilbert space $L^2(D)$ instead of the unit sphere in the Banach space $C(\overline D)\subset L^2(D)$.

For the Bayesian approach to level set inversion, we use a prior for the level set function $g$ in the form of a series expansion 
\begin{equation}
	\label{eq:KLE}
	g= g(\widebar{X}) = \sum_{i=1}^\infty \widebar{X}_i \phi_i
\end{equation}
where the $\phi_i \in C(\overline{D})$ are given and form an orthonormal system in $L^2(D)$, and the $\widebar{X}_i$ are random coefficients such that $\widebar{X} = (\widebar{X}_i)_{i\in\Naturals} \in \Sphere(\ell^2)$ almost surely; see e.g. \citep{IglesiasLuStuart2016, DunlopEtAl2017}. 
Then we have $g(\widebar{X}) \in \Sphere(L^2(D))$ almost surely.
A common choice for the system $\set{\phi_i}{ i \in \Naturals}$ are the eigenfunctions of a covariance operator $C\colon L^2(D) \to L^2(D)$ given by $C\phi(t) = \int_D c(s,t) \, \phi(s) \, \rd s$ and a continuous covariance function $c\in C(\overline{D} \times \overline{D})$.
The latter then yields $\phi_i \in C(\overline{D})$ for every $i\in\Naturals$, and by Mercer's theorem we have $g(\widebar{X}) \in C(\overline{D})$ almost surely.
In summary, we obtain a reformulation of the original level set inverse problem, as the problem of inferring the sequence $\xx\in \Sphere(\ell^2)$ that corresponds to the data-generating level set function $g(\xx)$, given the noisy data $y=O\circ \Upsilon \circ \mathcal{L}(g(\xx)) + \eta$.
For the prior $\mu_0$ on $\Sphere(\ell^2)$ in \eqref{eq:push_post} we consider the \ac{ACG} measure $\mu_0 = \ACG(\Lambda)$ where $\Lambda = \mathrm{diag}(\lambda_i\colon i \in \Naturals)$ involves the eigenvalues $(\lambda_i)_{i\in\Naturals}$ of the covariance operator $C$, i.e., $C\phi_i = \lambda_i\phi_i$. The associated distribution of $g(\widebar{X})$, $\widebar{X}\sim \ACG(\Lambda)$, on $\Sphere(L^2(D))$ is then $\ACG(C)$.

\begin{remark}[Boundedness of $\PPhi$]
Given that $\eta \sim \Normal(0,\Sigma)$, the negative log-likelihood $\PPhi$ for Bayesian level set inversion on $\Sphere$ takes the form
\[
	\PPhi(\xx) \defeq \frac{1}{2} \Norm{\Sigma^{-1/2} \left(y - O\circ \Upsilon \circ \mathcal{L}(g(\xx))  \right) }^2.
\]
Since the range of $\mathcal{L}$ is bounded in $L^\infty(D)$, i.e.\ $\norm{\mathcal{L}(g)}_{L^\infty(D)} \leq \max_{i} |u_i|$, and since $\Upsilon$ and $O$ are locally Lipschitz continuous and bounded respectively, we obtain that also $\PPhi$ is bounded on $\Sphere(\ell^2)$.
Thus, the setting of Bayesian level set inversion satisfies the assumptions of \Cref{theo:Uniform_ergodic,theo:main}.
In particular, for finite-dimensional approximations of the Bayesian level set inversion problem that are obtained by truncating the expansion \eqref{eq:KLE} after $d$ terms, \Cref{alg:pCN_Sphere,alg:ESS_Sphere,alg:Zappa_MCMC} yield uniformly ergodic Markov chains on $\Sphere^{d-1}$ targeting the corresponding posterior $\mu^{(d)}$.
Here, the corresponding posterior $\mu^{(d)}$ on $\Sphere^{d-1}$ may be obtained according to \eqref{eq:push_post_finite} or \eqref{eq:post_finite}.
\end{remark}

\paragraph{Problem setting.}
We consider the elliptic problem \eqref{eq:PDE} on $D = [0,1]$ with Dirichlet boundary conditions $p(0) = 0$ and $p(1) = 2$.
For $u$ we assume a form as in \eqref{eq:BLI_u} based on a decomposition of $D$ into two subregions $D_0,D_1$ with corresponding values $u_0 = -2$ and $u_1 = 2$.

For the level set function $g$ such that $u = \mathcal L(g)$ we assume a series expansion \eqref{eq:KLE} based on the eigensystem of the 
covariance function
\begin{equation}\label{eq:cov_func_exam}
	c(s,t)
	=
	\left(1 + \sqrt 3 \frac{|t-s|}{0.1} \right) \, \exp\left(- \sqrt3 \frac{|t-s|}{0.1} \right),
\end{equation}
i.e.\ a Whittle--Mat\'ern covariance with variance $\sigma^2 = 1$, correlation length $\rho = 0.1$ and smoothness $\nu = 1.5$.
For computational reasons we truncate the representation \eqref{eq:KLE} after $d$ terms and then infer the $d$ coefficients $\bar x_1,\ldots, \bar x_d$ 
given noisy observations of $p(0.2)$, $p(0.4)$, $p(0.6)$, and $p(0.8)$.

The assumed noise model is $\eta \sim \Normal (0, \Sigma)$ where $\Sigma = \mathrm{diag}(\sigma^2_1,\ldots,\sigma^2_4)$ and $\sigma^2_i = p^\dagger(0.2i)/10$ where $p^\dagger$ denotes the ``true'' solution resulting from the ``true'' coefficient vector $x^\dagger = (1,2,3,4,5,1,1,1,0,\ldots, 0)$ in the Karhunen--Lo\`{e}ve expansion \eqref{eq:KLE} of $g^\dagger$.
For an illustration of $g^\dagger$, $u^\dagger$ and $p^\dagger$, see \Cref{fig:truth}.

\begin{figure}
	\centering \includegraphics[trim = 20mm 0mm 20mm 0mm, clip,width=0.8\textwidth]{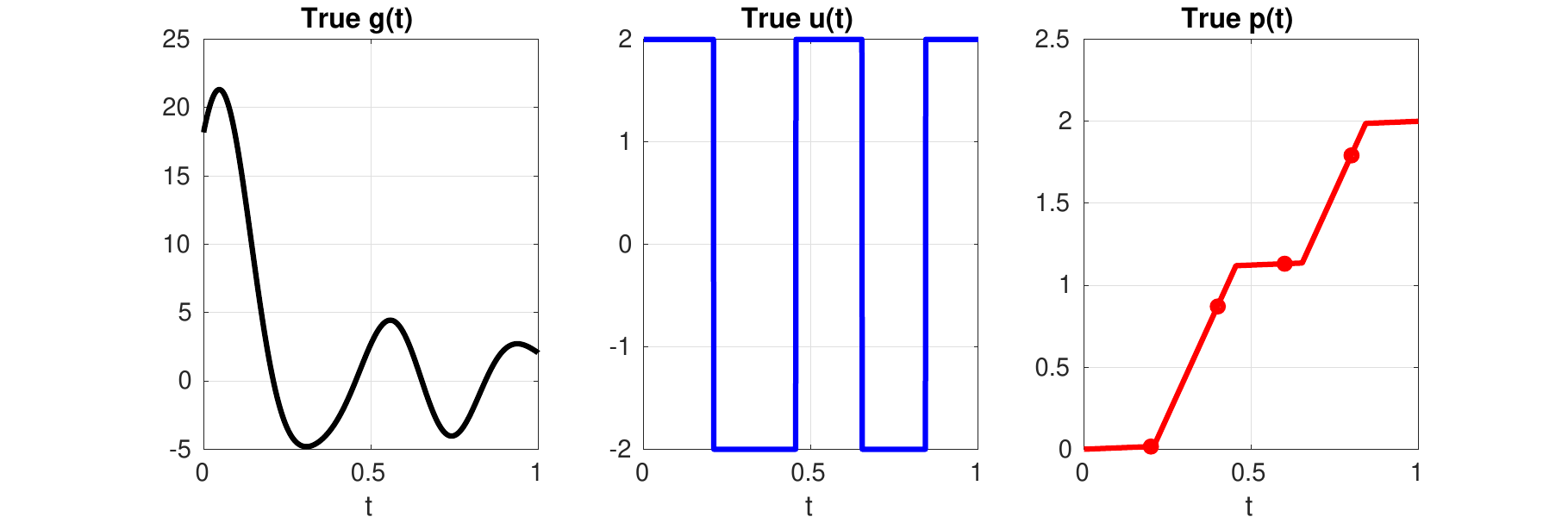} \\
	\caption{True $g^\dagger$, $u^\dagger$, $p^\dagger$, and true observations $o^\dagger = (p^\dagger(0.2),\ldots,p^\dagger(0.8))^\top$.}
	\label{fig:truth}
\end{figure}

Given $\eta \sim \Normal (0, \Sigma)$ the negative log-likelihood for observed data $y \in \Reals^4$ is then
\[
	\Phi(x)
	\defeq
	\frac{1}{2} \sum_{j=1}^4 \sigma_j^{-2} \left|y_j - F_j(x) \right|^2,
	\qquad
	x \in \Reals^d,
\]
where $F_j$ denotes the forward mapping $x=(x_i)_{i=1}^{d} \mapsto  g \mapsto u \mapsto p \mapsto p(0.2j)$.
As described in the previous subsection, $F_j(\alpha x)=F_j(x)$ for every $x \in\Reals^{d}$ and $\alpha >0$.
Thus, $\Phi$ is invariant under the radial projection map $\RadProjToSphere$, i.e.\ $\Phi\circ \RadProjToSphere=\Phi$, and we can consider Bayesian level set inversion on the sphere $\Sphere^{d-1}$ with corresponding prior
\[
	\mu_{0}
	=
	\ACG(\Lambda_d),
	\qquad
	\Lambda_d = \mathrm{diag}(\lambda_1,\ldots,\lambda_d).
\]
The posterior $\mu$ then takes the form \eqref{eq:push_post} with $\PPhi(\xx) = \Phi(\xx)$.
\Cref{fig:3D} illustrates the prior, the likelihood and the resulting posterior for dimension $d=3$.

\begin{figure}
	\includegraphics[width=0.32\textwidth]{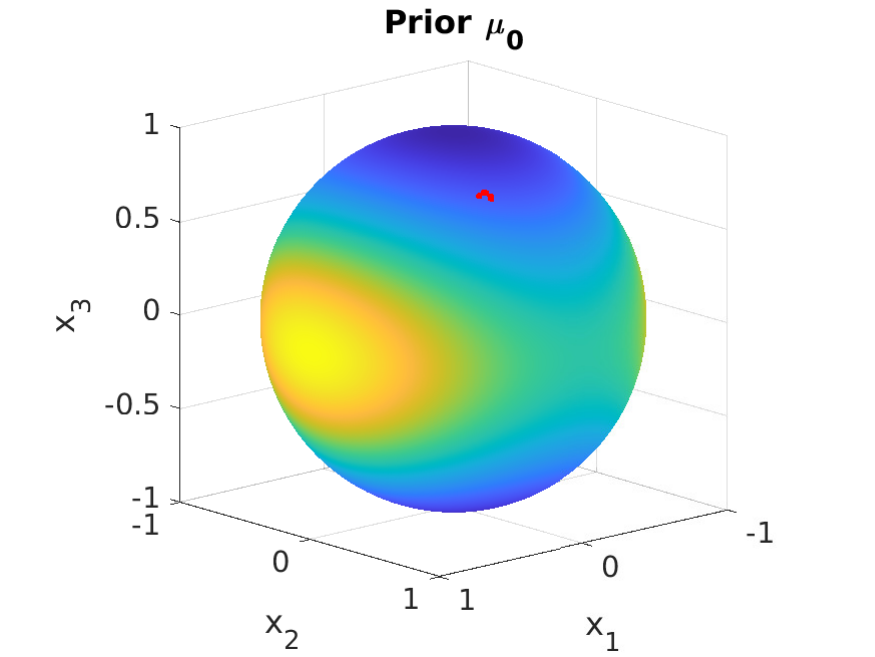}
	\includegraphics[width=0.32\textwidth]{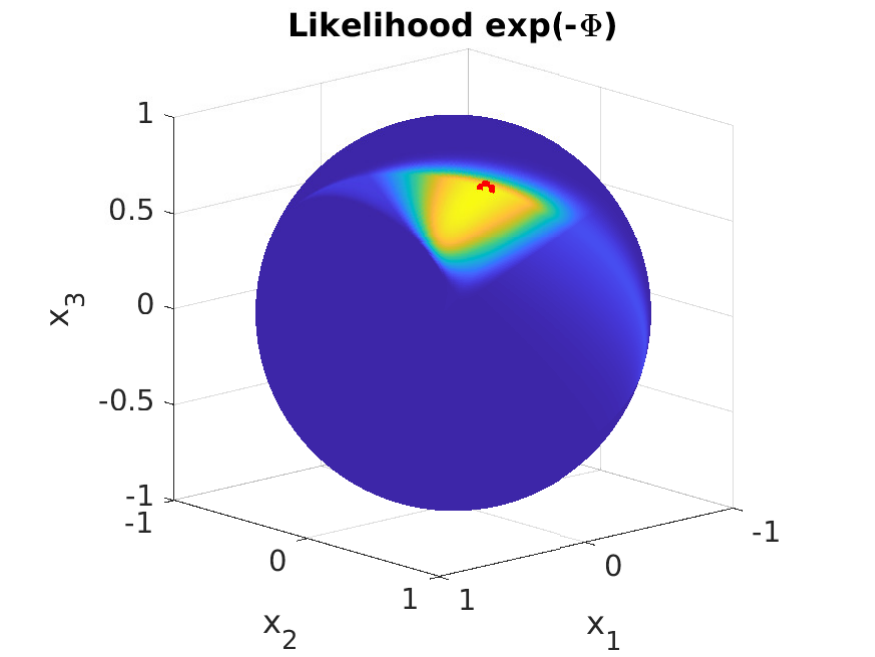}
	\includegraphics[width=0.32\textwidth]{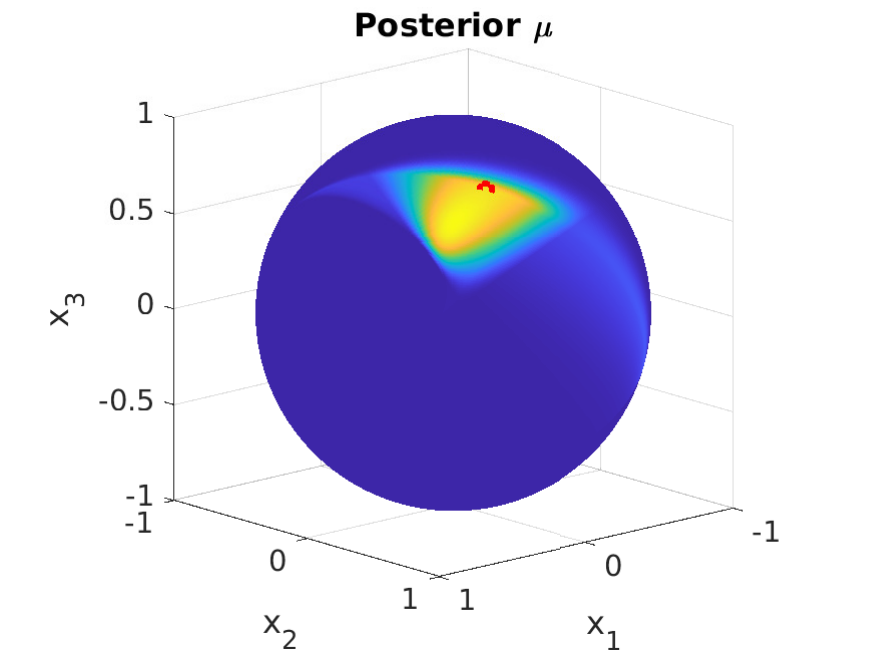}
	\caption{Prior, likelihood and posterior for dimension $d=3$; the red spot indicates the (projected) truth $\xx^\dagger$.}
	\label{fig:3D}
\end{figure}

\begin{remark}
The eigenpairs $(\lambda_j,\phi_j)$ of the covariance operator $C$ associated to $c$ in \eqref{eq:cov_func_exam} are computed numerically via a discretisation of $D$ and $C$ using a grid size of length $\delta t = 10^{-3}$.
The elliptic problem \eqref{eq:PDE} is solved numerically using the same discretisation.
Note that the solution $p$ of \eqref{eq:PDE} on $D=[0,1]$ is given by $\Upsilon(u)(t) = p(t) = 2 S_t(\e^{-u})/S_1(\e^{-u})$ with $S_t(f) = \int_0^t f(s) \rd s$.
We evaluate $S_t$ using the trapezoidal rule on the given grid.
\end{remark}

\paragraph{\ac{MCMC} on the sphere.}
We now apply the four \ac{MCMC} algorithms described in \Cref{sec:MCMC_on_Hilbert_sphere} in order to sample approximately from the posterior $\mu$ in various dimensions $d$.
In particular, we aim to compute the posterior expectation of the following quantity of interest:
\begin{align}\label{eq:BLSI_QoI}
	f(\xx)
	& =
	\left( \int_D \exp( - u(t,\xx)) \, \rd t \right)^{-1}
\end{align}
where $u(\quark, \xx) = \mathcal{L}(g(\xx))$.
We may interpret $f(\xx)$ as the effective homogenised permeability field over the one-dimensional domain $D$; see e.g.\ \citep[Section~2]{Alexanderian2015}.

First, we show in \Cref{fig:3D_samp} the thinned realisations $\xx_{100k}$, $k=1,\ldots,100$ of the Markov chains generated by the reprojected \ac{pCN}-\ac{MH} algorithm, the geodesic random walk-\ac{MH} algorithm based on \citep{MangoubiSmith2018}, and the \ac{MH} algorithm of \citep{Zappa2018} for $d=3$, subsampled every 100 steps.
\begin{figure}
	\includegraphics[trim = 5mm 0mm 10mm 0mm, clip, width=0.32\textwidth]{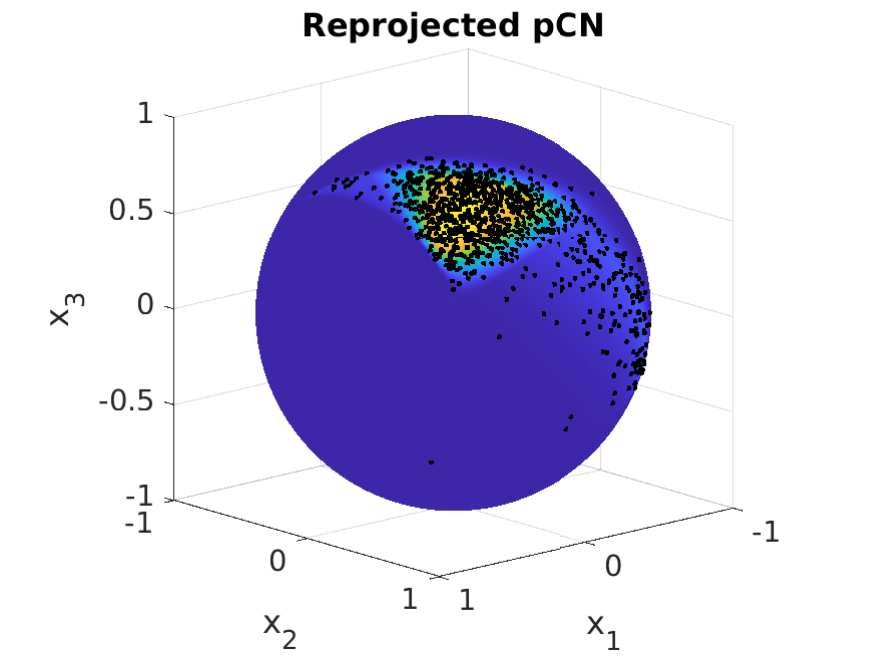}
	\includegraphics[trim = 5mm 0mm 10mm 0mm, clip, width=0.32\textwidth]{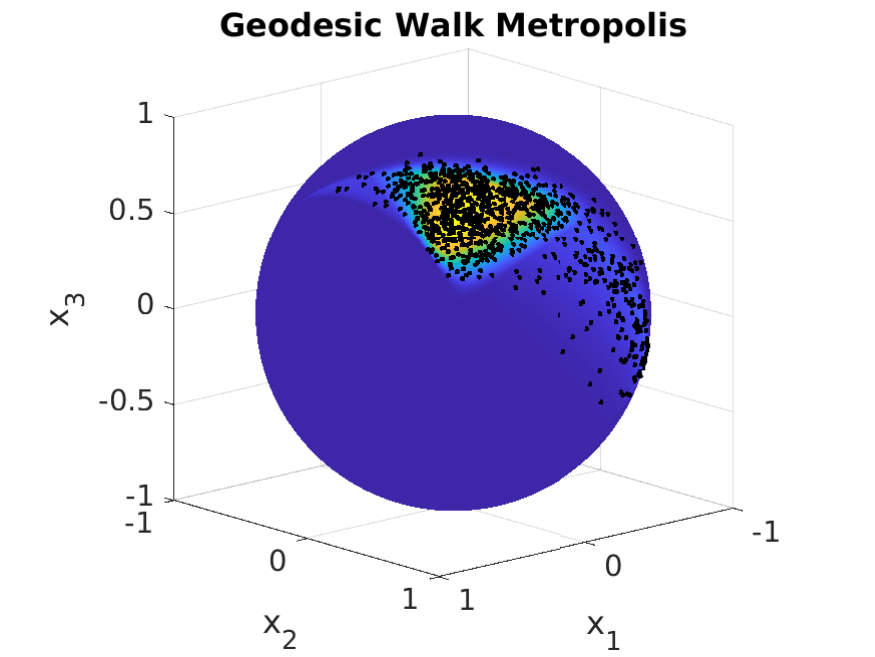}
	\includegraphics[trim = 5mm 0mm 10mm 0mm, clip, width=0.32\textwidth]{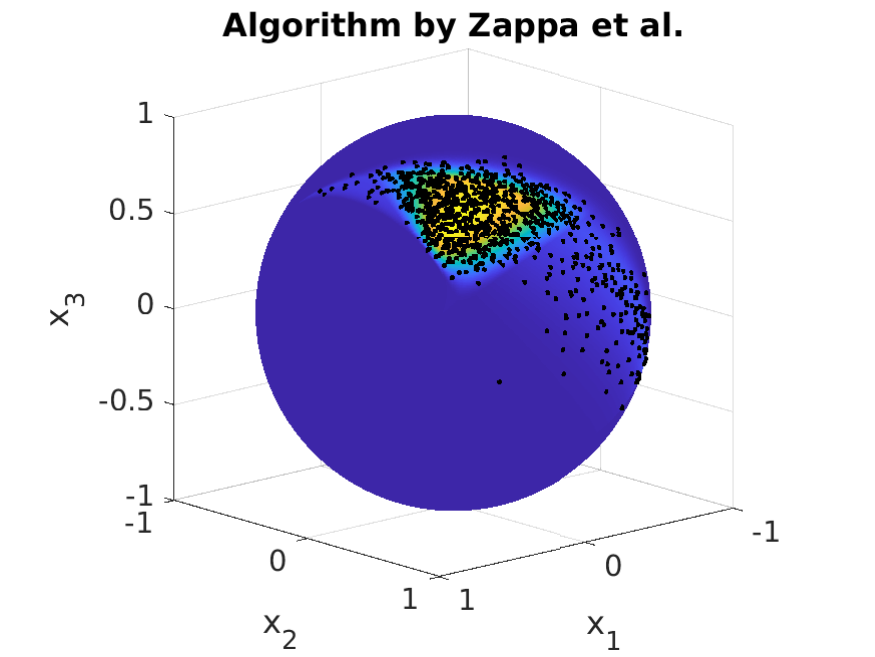}
    \caption{Thinned realisations $\xx_{100k}$, $k=1,\ldots,100$ of the Markov chains generated by \Cref{alg:pCN_Sphere,alg:geodesic_MCMC,alg:Zappa_MCMC} (left, centre, and right, respectively) for $d=3$, subsampled every 100 steps.}
	\label{fig:3D_samp}
\end{figure}

All three \ac{MH} algorithms were tuned to an average acceptance rate of roughly $23\%$.
We ran the algorithms for another $10^6$ iterations after a burn-in period of $5\cdot 10^4$ iterations.
All three runs yielded similar estimates for the posterior expectation of $f$.
We provide the corresponding estimate plus/minus the half-length of an $95\%$ confidence interval based on asymptotic variance estimates via the empirical autocorrelation functions of $(f(\xx_k))_{k}$:
\begin{align*}
\text{reprojected \ac{pCN}-\ac{MH}:} & \quad 0.420 \pm 1.299 \cdot 10^{-3}\\
\text{geodesic random walk-\ac{MH}:} & \quad 0.419 \pm 1.271 \cdot 10^{-3}\\
\text{MH by Zappa et al.:} & \quad 0.420 \pm 1.478 \cdot 10^{-3}.
\end{align*}
All three estimates exhibit similar accuracies.
Recall the root mean squared jump distance \wrt the Riemannian metric on the sphere given by
\begin{equation}\label{eq:RMSJD}
	\mathrm{RMSJD} \defeq  \sqrt{\frac 1{n-1} \sum_{k=1}^{n-1} d^2_{\Sphere}(\xx_k, \xx_{k+1}) }.
\end{equation}
In dimension $d=3$, the three \ac{MH} algorithms yielded similar estimates for the root mean squared jump distance:
\begin{align*}
\text{reprojected \ac{pCN}-\ac{MH}:} & \quad 0.202\\
\text{geodesic random walk-\ac{MH}:} & \quad 0.234\\
\text{\ac{MH} by Zappa et al.:} & \quad 0.185.
\end{align*}

Next, we tested all four algorithms, including now the reprojected \ac{ESS} algorithm, for increasing dimensions $d=10,20,40,80,160,320,640$.
In particular, we display the following quantities in \Cref{fig:comp1}:
\begin{enumerate}[label=(\roman*)]
\item
the estimated posterior expectation of $f$ given by the arithmetic mean of $(f(\xx_k))_{k}$;

\item
the estimated integrated autocorrelation time of the (approximately stationary) time series $(f(\xx_k))_{k}$ as a measure for the \ac{MCMC} error for computing the posterior expectation of $f$;
 \item the root mean squared jump distance as a measure of how well the Markov chain explores the sphere.
\end{enumerate}
An important observation from \Cref{fig:comp1} is that the two reprojected \ac{MCMC} methods show dimension-independent efficiency in terms of integrated autocorrelation time and root mean squared jump distance.
In contrast, the two \ac{MH} algorithms relying on the surface measure as reference measure --- namely, the geodesic random walk-\ac{MH} algorithm based on \citep{MangoubiSmith2018} and the \ac{MH} algorithm of \citep{Zappa2018} --- show a clear decrease in efficiency as the dimension $d$ of the state space increases.
In particular, these methods lead to less accurate estimates of the posterior mean; see the left plot in \Cref{fig:comp1}.
While it seems that the \ac{ESS} algorithm yields a higher efficiency, the higher efficiency comes at an increased cost: on average, the \ac{ESS} algorithm required $\approx 3.8$ tries until it hit the level set.
Thus, the computational cost of the \ac{ESS} algorithm was roughly four times higher than the computational cost of the \ac{pCN} algorithm.

\begin{figure}
	\includegraphics[width=0.32\textwidth]{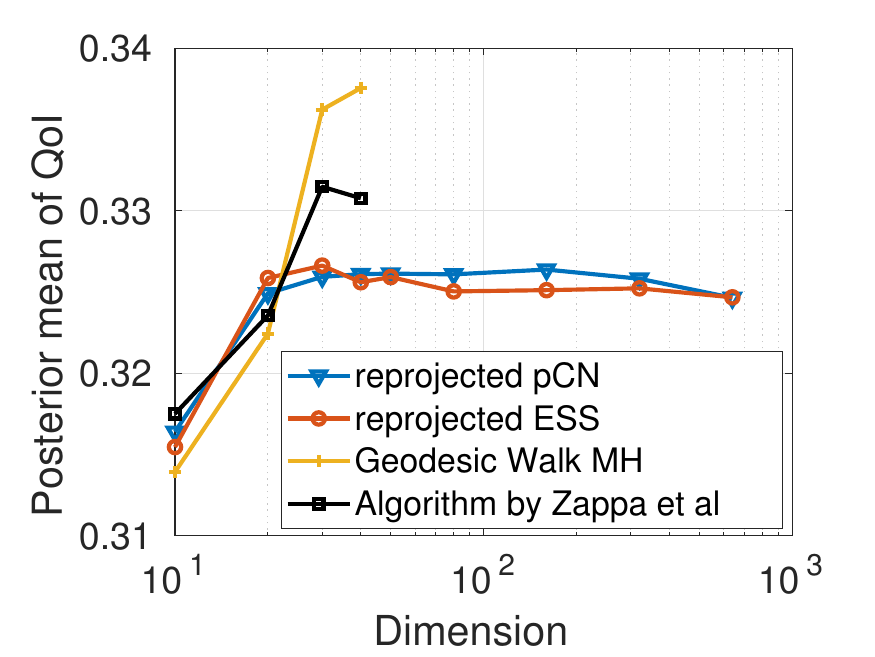}
	\hfill
	\includegraphics[width=0.32\textwidth]{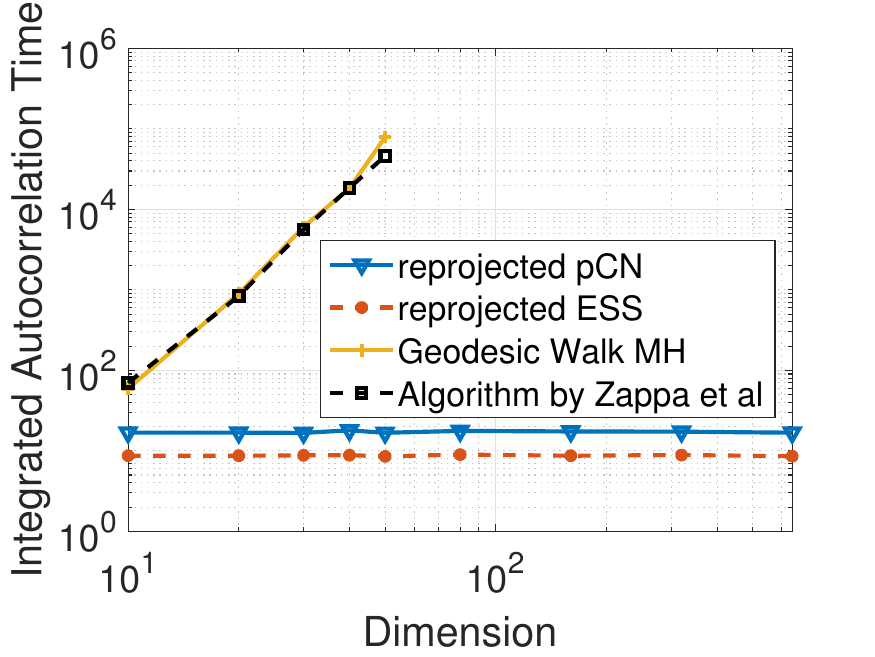}
	\hfill
	\includegraphics[width=0.32\textwidth]{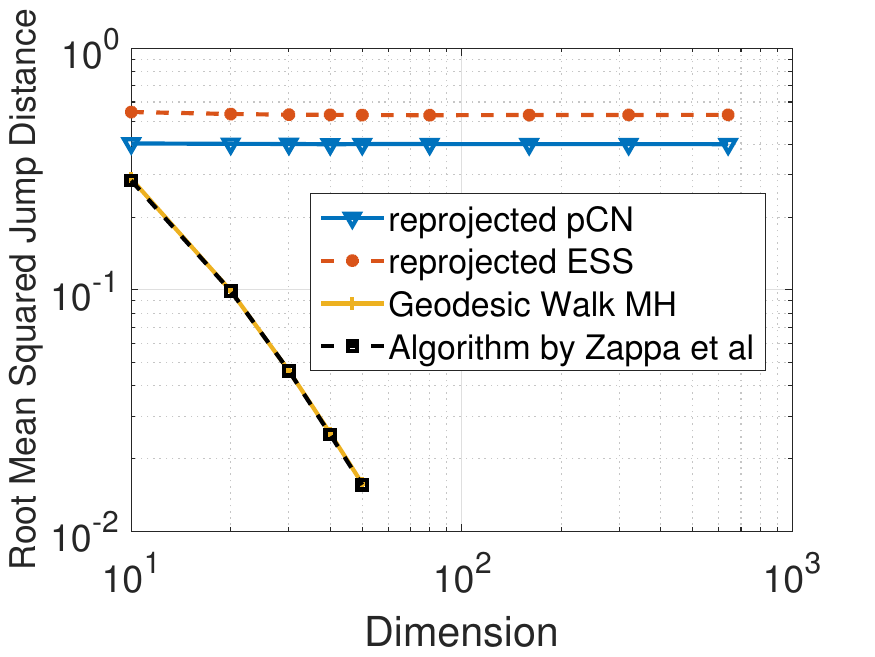}
	\caption{Estimated posterior mean (left), integrated autocorrelation time (middle) for $f$ as in \eqref{eq:BLSI_QoI} and root mean squared jump distance (right) for \Cref{alg:pCN_Sphere,alg:ESS_Sphere,alg:geodesic_MCMC,alg:Zappa_MCMC} as applied to the Bayesian level set inverse problem.}
	\label{fig:comp1}
\end{figure}

\subsection{Bayesian density estimation}
\label{ssec:numerical_density_estimation}
We now consider a second application of Bayesian inference on high-dimensional spheres.
We follow the approach of \citet{Holbrook2020} to nonparametric Bayesian density estimation: Given data $y_1,\ldots,y_n \in D$, for $D$ a bounded smooth domain in $\Reals^k$, we infer the Lebesgue probability density function $p$ of the data, where $p$ belongs to
\[
	\mathrm{P} \defeq \Set{ p\colon D \to [0,\infty) }{ \int_D p(y) \, \rd y =1 }.
\]
The set $\mathrm{P}$ is the unit simplex in the Banach space $L^1(D)$.
Instead of inferring $p$ directly, we instead infer the \emph{square root $g = \sqrt{p}$} of $p$, where $g$ belongs to
\[
	\mathrm{Q} \defeq \Set{ g\colon D \to \Reals}{ \int_D g^2(y) \, \rd y =1 }.
\]
The set $\mathrm{Q}$ coincides with the unit sphere of the function space $L^2(D)$.

\paragraph{Data.}
We use the British coal mine disaster data set that was studied in \citet{Holbrook2020}.
This data set consists of the dates of $191$ disasters recorded between March 1851 and March 1962 which can be found in \citet[Data set 204]{HandEtAl1994}.
We aim to estimate the underlying probability density between the years 1850 and 1965.
However, for computations and simplicity, we scale the data to lie within $D=[0,1]$ by a suitable affine transformation.
Thus, we would like to infer a square root density $g \in \Sphere=\mathrm{Q}$, where $\Sphere$ denotes the unit sphere of $L^2([0,1])$.

\paragraph{Prior and posterior.}
For numerical discretisation and constructing a prior model for $g$, we expand $g$ with respect to a suitable orthonormal system $\set{\phi_i}{ i \in \Naturals}$ in $L^2(D)$, $D = [0,1]$.
Here, we choose the same system as in \citep{Holbrook2020} which is based on a mean-zero Gaussian process model for $g$ with a Whittle--Mat\'ern covariance (cf.\ \eqref{eq:cov_func_exam}):
\begin{equation}
	\label{eq:KLE_BDE}
	g
	=
	g(\bar X)
	=
	\sum_{i=1}^\infty \bar X_i\, \phi_i,
	\qquad
	\phi_1 \equiv 1, \;\;
	\phi_i(y) = \sqrt{2} \cos(\pi\,(i-1)\,y),\; i > 1,
\end{equation}
where $\bar X = (\bar X_i)_{i\in\Naturals}$ are suitable random coefficients almost surely belonging to the unit sphere in $\ell^2$, such that the resulting $g = g(\bar X)$ belongs to the unit sphere in $L^2([0,1])$.
Here, we assume as prior for $\bar X$ that $\bar X \sim \ACG(\Lambda)$ where $\Lambda = \mathrm{diag}(\lambda_i\colon i \in \Naturals)$ with $\lambda_i = \sigma^2 (\kappa + \pi^2 (i-1)^2)^{-r}$, $i\geq1$.
We choose $\sigma = 0.5$, $\kappa = 0.1$ and $r = 1$ for our experiments.
We chose these parameter values because of their similarity to the values used by \citet{Holbrook2020}.
The resulting prior for $g$ is then $\ACG(C)$ where $C\colon L^2(D)\to L^2(D)$ denotes the covariance operator $C v(y) = \int_D c(y,y') v(y') \, \rd y'$ using the corresponding covariance function
\[
	c(y,y')
	=
	\sum_{i=0}^\infty \lambda_i\, \phi_i(y) \, \phi_i(y').
\]
Given $g \in \mathrm{Q}$, the likelihood $L(y; g)$ for observing the data  $y = (y_1,\ldots,y_n) \in D^n$ is
\[
	L(y; g)
	=
	\prod_{i=1}^n g^2(y_i)
	=
	\exp\left(\sum_{j=1}^n \log(g^2(y_j) ) \right)
	=
	\exp\left(2 \sum_{j=1}^n \log(\absval{g(y_j)}) \right).
\]
Thus, using the series representation in \eqref{eq:KLE_BDE}, we obtain as likelihood for $y$ given coefficients $\xx $ in the unit sphere of $\ell^2$,
\[
	L(y; \xx)
	=
	\exp\left(- \PPhi(\xx)\right),
	\qquad
	\PPhi(\xx)
	\defeq
	- 2  \sum_{j=1}^n \log\left( \Absval{ \sum_{i=0}^\infty \xx_i \, \phi_i(y_j) } \right).
\]
Note that $\PPhi\colon \ell^2 \to \Reals$ is bounded.
Given the data $y$, the resulting posterior for the coefficients $\widebar{X}$ follows the form \eqref{eq:push_post} with $\mu_0 = \ACG(\Lambda)$, and the assumptions of \Cref{theo:Uniform_ergodic} are satisfied.

A quantity of interest for this problem is the posterior expectation for the probability mass of $p$ between $0.435$ and $0.574$.
This quantity of interest is the probability of a coal mine disaster between the years 1900 and 1916.
It can be written as
\begin{align}\label{eq:BDE_QoI}
	f(\xx)
	& \defeq
	\int_{0.435}^{0.574} g^2(y, \xx) \, \rd y
	=
	\left( \sum_{i,k=1}^\infty w_{i,k} \xx_i \, \xx_k \right)^2
\end{align}
where $w_{i,k} \defeq \int_{0.435}^{0.574} \phi_i(y)\, \phi_k(y)\, \rd y$, i.e., $f$ is a quadratic function of $\xx$.

\paragraph{\ac{MCMC} simulations.}
We truncated the expansion in \eqref{eq:KLE_BDE} after $d$ terms for $d = 10$, $20$, $30$, $40$, $50$, $100$, $200$, $400$, and $800$ and sampled approximately from the resulting truncated posterior $\mu^{(d)}$ on $\Sphere^{d-1}$ for the coefficients $\xx^{(d)} = (\xx_1,\ldots,\xx_d)$ using the prior $\mu^{(d)}_0 = \ACG(\Lambda_d)$ with $\Lambda_d \defeq \mathrm{diag}(\lambda_1, \ldots, \lambda_d)$.
To this end, we applied the four \ac{MCMC} algorithms described in \Cref{sec:MCMC_on_Hilbert_sphere}, and used them to compute the expectation of the quantity $f$ in \eqref{eq:BDE_QoI} with respect to the truncated posterior $\mu^{(d)}$.
After a burn-in period of $10^5$ iterations, we ran the algorithms for $10^6$ iterations and compared their efficiency.
We quantified their efficiency in terms of the estimated integrated autocorrelation time for $f$ and the root mean squared jump distance.
We display the results in \Cref{fig:comp2}.
As in \Cref{fig:comp1}, the results exhibit dimension-independent efficiency of the two reprojected \ac{MCMC} methods, whereas the geodesic random walk-\ac{MH} algorithm
and the \ac{MH} algorithm of \citet{Zappa2018} show a clear and drastic deterioration of efficiency as $d$ increases.
\begin{figure}
	\includegraphics[width=0.32\textwidth]{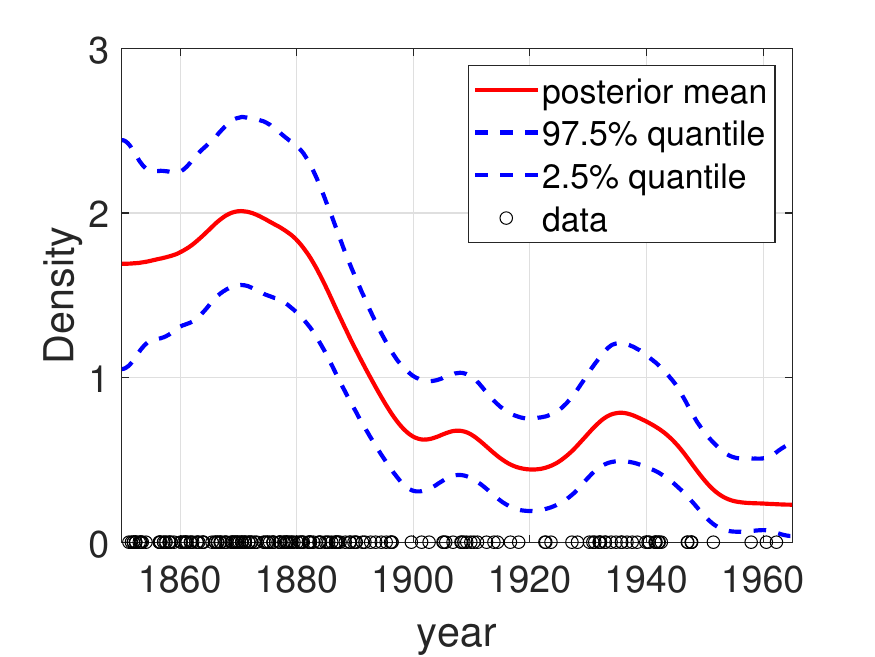}
	\hfill
	\includegraphics[width=0.32\textwidth]{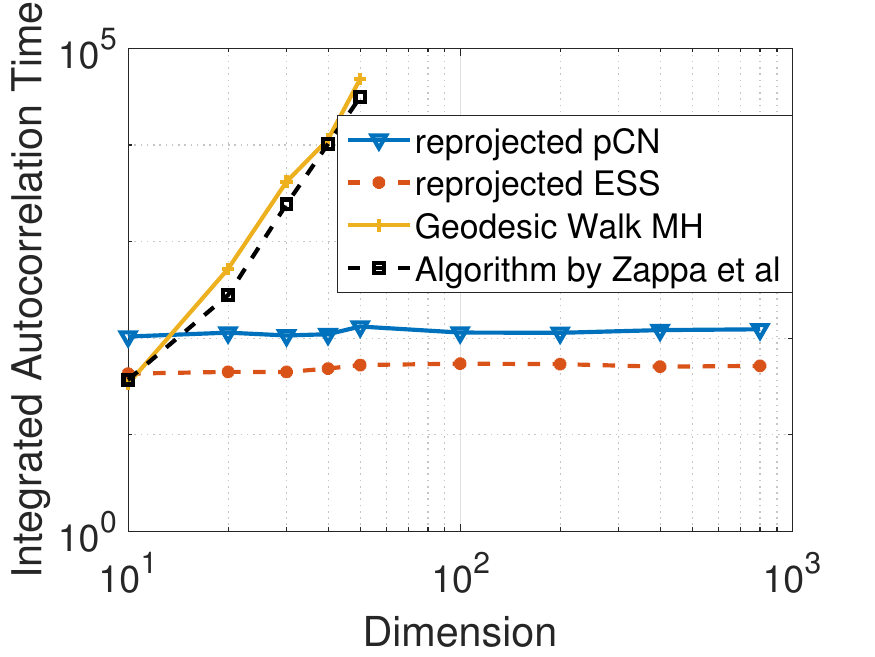}
	\hfill
	\includegraphics[width=0.32\textwidth]{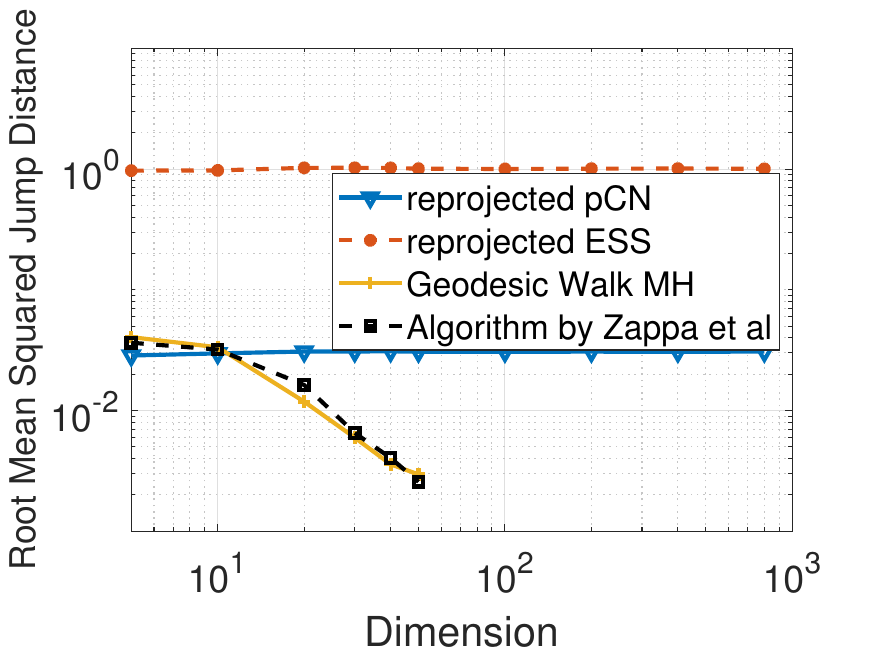}
	\caption{Estimated posterior mean and pointwise posterior quantiles for the density $p$ based on truncation of \eqref{eq:KLE_BDE} to the first $d=30$ summands (left), estimated integrated autocorrelation time for quantity of interest $f$ as in \eqref{eq:BDE_QoI} (middle), and root mean squared jump distance (right) for \Cref{alg:pCN_Sphere,alg:ESS_Sphere,alg:geodesic_MCMC,alg:Zappa_MCMC} as applied to the Bayesian density estimation problem.}
	\label{fig:comp2}
\end{figure}

\section{Closing remarks}
\label{sec:closing}

In this paper, we proposed efficient \ac{MCMC} algorithms for sampling target measures on high-dimensional spheres that are absolutely continuous \wrt an ACG prior.
Our algorithms exploit the structure of the \ac{ACG} prior by lifting the sampling problem on the sphere to a sampling problem in the ambient Hilbert space.
This allows us to apply existing \ac{MCMC} algorithms on linear spaces --- such as the \ac{pCN}-\ac{MH} algorithm --- for which there are theoretical results concerning dimension-independent efficiency.

Using the technique of push-forward Markov kernels, we then obtained transition kernels on the sphere that inherit many properties of the transition kernels in the ambient Hilbert space, e.g.\ reversibility \wrt the desired target measure.
Under fairly mild conditions, we showed the uniform ergodicity of Markov chains generated by our algorithms, and provided theoretical arguments for the dimension independence of their spectral gaps.

Using binary classification and Bayesian density estimation as test problems, we compared the performance of our methods to that of \ac{MH} algorithms based on the geodesic random walk proposal of \citet{MangoubiSmith2018} and the approach of \citet{Zappa2018}.
Our results illustrated the robustness of our algorithms as the dimension of the state space increased.
In comparison, the statistical efficiency of the two other existing algorithms decreased as the dimension of the state space increased.

Based on our work, several interesting questions for future research remain.
A theoretical analysis of the dimension independence of the (reprojected) \ac{ESS} transition kernel remains an open issue.
Here Markov chain comparison techniques --- as have been developed by, for example, \citet{Peskun1973}, \citet{Andrieu2016}, and \citet{RuSpru2018} --- may be useful for establishing the inheritance of a spectral gap from the \ac{pCN}-\ac{MH} transition kernel to the \ac{ESS} transition kernel.
Additionally, it seems promising to modify our algorithms so that they can be applied to sample from target measures \wrt other common priors such as Bingham distributions, by using the acceptance-rejection approach of \citet{Kent2018}, for example.

\section*{Acknowledgements}
\addcontentsline{toc}{section}{Acknowledgements}

The research of HCL has been partially funded by the \ac{DFG} --- Project-ID \href{https://gepris.dfg.de/gepris/projekt/318763901}{318763901} --- SFB1294.
BS has been supported by the \ac{DFG} project \href{https://gepris.dfg.de/gepris/projekt/389483880}{389483880}.
DR gratefully acknowledges partial support of the \href{http://fbms.math.uni-goettingen.de/index.php?id=15}{Felix Bernstein Institute for Mathematical Statistics in the Biosciences} and the \ac{DFG} within project \href{https://gepris.dfg.de/gepris/projekt/456842089}{432680300} --- SFB
1456.
TJS has been partially supported by the Freie Universit\"at Berlin and the Zuse Institute Berlin within the Excellence Initiative of the \ac{DFG}, and by \ac{DFG} projects \href{https://gepris.dfg.de/gepris/projekt/390685689}{390685689} and \href{https://gepris.dfg.de/gepris/projekt/415980428}{415980428}.
The authors wish to thank Stephan Huckemann and Michael Habeck for helpful comments and, in particular, Andre Wibisono for the inspiring discussion which led to the proof of \Cref{theo:main}.

\bibliographystyle{abbrvnat}
\addcontentsline{toc}{section}{References}
\bibliography{acg-sphere-mcmc}

\appendix

\section{Gaussian measures and angular central Gaussian measures}\label{sec:Gaussian_ACG}
\subsection{Gaussian measures on Hilbert spaces}
\label{ssec:Gaussian}

We first briefly recall some basic notions related to measures, especially Gaussian measures, on Hilbert spaces.
A standard reference in this area is the book of \citet{Bogachev1998}.

In the following, we consider a separable Hilbert space $\HH$ with norm $\norm{ \quark }$ and inner product $\innerprod{ \quark }{ \quark }$.
Suppose $\mu \in \probson{\HH}$ with finite second moment $\int_{\HH} \norm{ x }^{2} \, \mu (\rd x)$ is given.
Then the \emph{mean element} $a \in \HH$ and \emph{covariance operator} $C \colon \HH \to \HH$ of $\mu$ are determined by
\begin{align}
	0 & = \int_{\XX} \innerprod{ v }{ x - a } \, \mu (\rd x) & & \text{for all $v \in \HH$,} \\
	\innerprod{ C u }{ v } & = \int_{\XX} \innerprod{ u }{ x - a } \innerprod{ v }{ x - a } \, \mu (\rd x) & & \text{for all $u, v \in \HH$.}
\end{align}
When $a = 0$, the measure $\mu$ is said to be \emph{centred}.

In this paper, we focus on \emph{Gaussian measures} $\mu = \Normal (a, C)$ on $\HH$.
Such measures are equivalently determined by the property that every one-dimensional linear image is Gaussian on $\Reals$, i.e.\ identifying $\ell \in \HH$ with the continuous linear functional $\HH \to \Reals$, $x \mapsto \innerprod{\ell}{x}$,
\[
	\mu = \Normal (a, C) \in \probson{\HH} \iff \forall \ell \in \HH,\; \ell_{\sharp} \mu = \Normal ( \innerprod{\ell}{a} , \innerprod{C\ell}{\ell} ) \in \probson{\Reals} ;
\]
or by the form of the characteristic function, i.e.
\[
	\mu = \Normal (a, C) \in \probson{\HH} \iff \forall v \in \HH,\; \int_{\XX} \exp (i \innerprod{ v }{ x } ) \, \mu( \rd x ) = \exp \bigl( i \innerprod{ v }{ a } - \tfrac{1}{2} \innerprod{ C v }{ v } \bigr) .
\]

The covariance operator $C$ of a probability measure on $\HH$ is always self-adjoint and positive semi-definite, and so its eigenvalues are all real and non-negative.
Furthermore, the assumption of finite second moment implies that these eigenvalues are summable, i.e.\ $C$ is a trace-class operator.

Given any self-adjoint and positive-definite operator $C \colon \HH \to \HH$, we define for $x, y \in \rg(C^{1/2})$ the weighted inner product and norm
\begin{align}
	\label{eq:precision_inner_product}
	\innerprod{ x }{ y }_{C} & \defeq \innerprod{ C^{-1/2} x }{ C^{-1/2} y } , \\
	\label{eq:precision_norm}
	\norm{ x }_{C} & \defeq \sqrt{\innerprod{ x }{ x }_{C}} ,
\end{align}
If $C$ is the covariance operator of a Gaussian measure $\mu = \Normal(a, C)$, then $\rg(C^{1/2})$ is the \emph{Cameron--Martin space} of $\mu$, and \eqref{eq:precision_inner_product} and \eqref{eq:precision_norm} are often called the \emph{precision} inner product and norm respectively.

The topological support of $\Normal(0, C)$ --- i.e.\ the smallest closed set of full measure --- is the closure in $\HH$ of $\rg(C)$ \citep[Theorem~3.6.1]{Bogachev1998}.
A sufficient condition for the density of $\rg(C)$ in $\HH$ is the strict positivity of $C$, i.e.\ that it has no null eigenvalue.
The containments $\rg(C) \subseteq \rg(C^{1/2}) \subseteq \HH$ always hold, and are strict when the Cameron--Martin space $\rg(C^{1/2})$ has infinite dimension.

\subsection{Angular central Gaussian measures on spheres in Hilbert spaces}
\label{ssec:ACG}

We denote the unit sphere in a separable Hilbert space $\HH$ by
\[
	\Sphere \defeq \set{ x \in \HH }{ \norm{ x } = 1 }.
\]
If $\HH = \Reals^{d}$, then we emphasise the dimension and denote the unit sphere by $\Sphere^{d - 1}$.
The sphere $\Sphere$ is equipped with the following metric $d_{\Sphere} \colon \Sphere\times\Sphere \to [0,\pi]$:
\begin{equation}
	\label{eq:dS}
	d_{\Sphere}(\xx,\yy)
	\defeq
	\arccos \left( \innerprod{ \xx }{ \yy } \right)
	=
	2\arcsin\left(\frac{1}{2}\norm{\xx-\yy}\right)
	\qquad
	\xx,\yy \in \Sphere,
\end{equation}
where the second identity follows from elementary trigonometry.
The definition of $d_{\Sphere}$ by the arcsine function also extends to the unit sphere $\Sphere$ in general Banach spaces.
Moreover, it yields the Lipschitz equivalence --- and hence the topological equivalence --- of $d_{\Sphere}$ and the metric on $\Sphere$ induced by the norm $\norm{ \quark }$ of the ambient space $\HH$:
\begin{equation}
	\label{eq:dS_equiv}
	\norm{ \xx - \yy }
	\leq
	d_{\Sphere}(\xx,\yy)
	\leq \frac \pi2 \norm{ \xx - \yy }
	\qquad
	\forall \xx,\yy \in \Sphere,
\end{equation}
since $g(r) = 2 \arcsin\left(\frac{1}{2} r\right) / r$ varies between $1$ and $\pi/2$ for $r\in[0,2]$.
Thus, the metric $d_{\Sphere}$ is topologically equivalent to the norm of the ambient space on $\Sphere$.
We record this result in \Cref{lem:dS_topology_equiv_rel_norm_topology} below.
According to \citet[Proposition~3.3(ii)]{Kechris1995}, a closed subset of a Polish space is always itself Polish in the relative (subspace) topology.
Thus, $(\Sphere, d_{\Sphere})$ is a Polish space and
$\Borel{ \Sphere } = \set{ B \cap \Sphere }{ B \in \Borel{\HH} }$.

\begin{lemma}
	\label{lem:dS_topology_equiv_rel_norm_topology}
	The topology on $\Sphere$ generated by $d_{\Sphere}$ coincides with the relative topology on $\Sphere$ generated by the norm $\norm{\quark}$ of $\HH$.
\end{lemma}
\begin{proof}
	This is a consequence of the Lipschitz equivalence of the metrics \eqref{eq:dS_equiv}.
\end{proof}

Fix an arbitrary $\bar{z}\in\Sphere$.
Recall the radial projection to the sphere defined in \eqref{eq:radial_projection_map_T},
\begin{equation*}
	\RadProjToSphere\colon \HH \to\Sphere,\quad \RadProjToSphere(x) \defeq \begin{cases}
		\frac{x}{\norm{ x }}, & \text{if $x \neq 0$,}
		\\
		\bar{z}, & \text{if $x = 0$.}
	\end{cases}
\end{equation*}
The choice of $\bar{z}$ is not important in what follows, but we fix $\bar{z}$ in order to ensure that $\RadProjToSphere$ is a measurable mapping from $\HH$ into $\Sphere$.

Recall that the \ac{ACG} measure on $\Sphere$ associated to a centred Gaussian measure $\Normal(0, C)$ on $\HH$ is given by
\[
	\mu \defeq \RadProjToSphere_{\sharp} \Normal (0, C)
\]
and denoted by $\mu = \ACG (C)$.
In other words, under $\ACG (C)$, each $E \in\Borel{\Sphere}$ is assigned the Gaussian $\Normal(0, C)$-measure of the cone $\set{ \alpha u \in \XX }{ \alpha > 0, u \in E }$.
Note that $\ACG(C) = \ACG(\lambda C)$ for $\lambda \neq 0$ and so it can sometimes be useful to fix a normalisation for $C$, e.g.\ by taking its leading eigenvalue to be unity.

\begin{corollary}
	\label{cor:openness_of_cones}
	Let $C_{E} \defeq \set{ \alpha e }{ \alpha > 0, e \in E } \subseteq \HH$ denote the cone spanned by a subset $E \subseteq \Sphere$.
	The set $E$ is open in $(\Sphere, d_{\Sphere})$ if and only if $C_{E}$ is open in $(\HH, \norm{\quark})$.
\end{corollary}
\begin{proof}
	Suppose that $C_{E}$ is open in $\HH$.
	Then $E = C_{E} \cap \Sphere$ is, by \Cref{lem:dS_topology_equiv_rel_norm_topology}, open in $(\Sphere, d_{\Sphere})$.

	Let $\RadProjToSphere\vert_{\HH\setminus\{0\}}$ be the restriction of the radial projection $\RadProjToSphere$ defined in \eqref{eq:radial_projection_map_T} to $\HH\setminus\{0\}$.
	For $E$ in $(\Sphere,d_{\Sphere})$, $ (\RadProjToSphere\vert_{\HH\setminus\{0\}})^{-1}(E) = C_E$.
	Equip $\HH\setminus\{0\}$ with the subspace topology.
	Then $\RadProjToSphere\vert_{\HH\setminus\{0\}}$ is continuous, and if $E$ is open, then $C_E$ is open in $(\HH\setminus\{0\},\norm{\quark})$, and hence in $(\HH,\norm{\quark})$.
\end{proof}
Whenever $\Normal (0, C)$ has support equal to $\HH$, \Cref{cor:openness_of_cones} immediately implies that the induced $\ACG(C)$ measure has support equal to $\Sphere$;
thus, in view of \Cref{ssec:Gaussian}, $\ACG (C)$ is a strictly positive measure on $\Sphere$ whenever $C$ is positive definite.

In the case where $\HH = \Reals^{d}$ with the usual Euclidean norm and $C\in\Reals^{d \times d}$ is symmetric and positive definite, one can show that the density $\rho \colon \Sphere^{d - 1} \to [0, \infty)$ of $\mu=\ACG (C)$ \wrt the $(d - 1)$-dimensional Hausdorff measure on the sphere is
\begin{align*}
	\rho(\xx)
	& = \int_{0}^{\infty} \frac{ \exp(- \tfrac{1}{2} (r \xx) \cdot C^{-1} (r \xx)) }{ \sqrt{ \det (2 \pi C) } } r^{d - 1} \, \rd r \\
	& = \frac{ \Gamma(d / 2) }{ 2 (\tfrac{1}{2} \norm{ \xx }_{C}^{2})^{d / 2} \sqrt{ \det (2 \pi C) } } \\
	& = \frac{ \Gamma(d / 2) }{ 2 \pi^{d / 2} \sqrt{ \det C } } \norm{ \xx }_{C}^{- d} .
\end{align*}
In the second equation, we used the fact that, for $a > 0$ and $n \in \Naturals$, $\int_{0}^{\infty} a^{- r^{2}} r^{n - 1} \, \rd r = \frac{\Gamma(n / 2)}{2 (\log a)^{n/2}}$
\citep[e.g.][Equation~(1)]{Tyler1987}.

\begin{remark}
	\label{rem:ACG_on_projective_spaces}
	The above definitions and properties can be readily adapted to the image of a centred Gaussian measure on a projective space, i.e.\ the quotient of $\HH$ by the equivalence relation $u \sim v \iff u = \lambda v$ for some $\lambda \neq 0$, i.e.\ $\Sphere$ with antipodal points identified.
	Note that, while $\ACG(C)$ on $\Sphere$ is always a multimodal distribution with at least two modes, its image on the projective space can be unimodal, which may be desirable in applications.
\end{remark}

\section{Markov chain Monte Carlo}
\label{sec:MCMC}

Markov chain methods provide a standard tool for approximate sampling of complicated distributions, such as posterior distributions appearing in the Bayesian analysis of data.
We recall here notions related to \ac{MCMC} on a general state space $\XX$ equipped with a target probability measure $\mu \in \probson{\XX}$, the Metropolis--Hastings and slice sampling paradigms, and make some particular points about \ac{MCMC} on infinite-dimensional Hilbert spaces.

\subsection{General notions}\label{sec:MCMC_general}

By a \emph{Markov kernel} on a topological space $\XX$ we mean a function $K \colon\XX \times \Borel{\XX} \to [0,1]$ such that $K (x, \quark) \in \probson{\XX}$ for each $x \in \XX$, and $K (\quark, A)$ is a measurable function for each $A \in \Borel{X}$.
A sequence of random variables $(X_n)_{n\in\mathbb{N}}$, mapping from $(\Omega,\mathcal{A},\Prob)$ to $\XX$, is a (\emph{time-homogeneous}) \emph{Markov chain} with \emph{transition kernel} $K$ on $\XX$ if
\[
	\Prob(X_{n+1}\in A
	\mid X_1,\dots,X_n) = \Prob(X_{n+1}\in A\mid X_n) = K(X_n,A), \quad
n\in\mathbb{N},\,A\in\Borel{\XX},
\]
where $K$ is a Markov kernel\footnote{In this paper, we shall distinguish between Markov kernels and transition kernels.
A transition kernel is a Markov kernel associated to a time-homogeneous Markov chain.
}.
We abuse notation and also use $K$ to denote the transition operator induced by the kernel $K$;
the transition operator acts on functions $f \colon \XX \to \Reals$ via
\begin{equation}
	\label{eq:transition_operator}
	(K f) (x) \defeq \int_{\XX} f(y) \, K(x, \rd y) = \ev{ f(X_{n + 1}) \mid X_{n} = x }\quad \text{for $x \in \XX$.}
\end{equation}

The idea of \ac{MCMC} is to construct a Markov chain $(X_n)_{n\in\mathbb{N}}$ with transition kernel $K$ such that the distribution of $X_n$ converges to $\mu$ as $n\to\infty$.
Ideally, this convergence will be ``fast'' and the correlation between successive random variables $X_{n}$ and $X_{n+1}$ will also be weak.

A necessary condition for a transition kernel $K$ to have $\mu$ as a limiting distribution is that $\mu$ is an \emph{invariant distribution} of $K$, i.e.
\begin{equation}
	\label{eq:invariant_distribution}
	\mu(A)
	=
	\mu K(A)
	\defeq
	\int_{\XX} K(x,A) \, \mu(\rd x), \quad \text{for all $A\in\Borel{\XX}$.}
\end{equation}
\emph{Reversibility} (or \emph{detailed balance}) of a transition kernel $K$ on $\XX$ \wrt $\mu$ refers to the property that
\begin{equation}
	\label{eq:detailed_balance}
	K(x,\rd y) \mu(\rd x) = K(y,\rd x) \mu(\rd y),
\end{equation}
i.e.\ that the measure $K(x,\rd y) \mu(\rd x)$ on $\XX\times\XX$ is symmetric.
If \eqref{eq:detailed_balance} holds we say that $K$ is $\mu$-reversible.
Reversibility of a transition kernel $K$ \wrt $\mu$ implies that $K$ has $\mu$ as an invariant distribution, although the converse is generally false.

If $\mu$ is an invariant distribution of $K$,
and if some non-restrictive regularity conditions such as $\varphi$-irreducibility and Harris recurrence hold, then a strong law of large numbers holds, i.e.
\[
	\lim_{n \to \infty} \frac{1}{N} \sum_{n = 1}^{N} f(X_{n}) = \int_{\XX} f(x) \, \mu (\rd x) \quad \Prob\text{-a.s.},
\]
for any $\mu$-integrable $f\colon\XX\to\Reals$ \citep{MeTw09,AsGl11}.
This shows that the ``\ac{MCMC}-time average'' $\frac{1}{N} \sum_{n = 1}^{N} f(X_{n})$ can be used to approximate the mean of $f$ \wrt the distribution of interest $\mu$.
For more quantitative statements the spectral gap of a Markov chain (or its transition kernel) is a crucial quantity.
For a transition kernel $K$ that is reversible \wrt $\mu$ (inducing the transition operator $K$ via \eqref{eq:transition_operator}) we define
\begin{equation}
	\label{eq:gap}
	\gap_{\mu}(K) \defeq 1 - \norm{ K }_{L^{2}_{0,\mu} \to L^{2}_{0,\mu}},
\end{equation}
where we recall that $L^2_{0,\mu}\defeq \set{f\in L^2(\mu)}{ \int f(u) \, \mu(\rd u)=0}$ and where $\norm{ K }_{L^{2}_{0,\mu} \to L^{2}_{0,\mu}}$ denotes the norm of the operator $K$, which we view as an element of the space of bounded linear operators from $L^{2}_{0,\mu}$ to itself.
An $L^2_{\mu}$-spectral gap, that is, $\gap_{\mu}(K)>0$, leads to desirable properties of a Markov chain $(X_n)_{n\in\mathbb{N}}$ with transition kernel $K$.
For instance, it implies a central limit theorem, see e.g.\ \cite[Section~22.5]{DoucEtAl2018} and the relevant references therein.
In particular, an explicit lower bound on $\gap_{\mu}(K)$ leads to an
estimate of the total variation distance and a mean squared error bound of the time average $\frac{1}{N} \sum_{n = 1}^{N} f(X_{n})$.
More precisely, it is well known, e.g.\ by virtue of \cite[Lemma~2 and Lemma~3]{novak2014computation}, that
\begin{equation}
	\label{eq:geometric_ergodic}
	\norm{ \xi K^{n} - \mu }_{\TV}
	\leq (1-\gap_{\mu}(K))^n
	\left \norm{ \frac{\rd \xi }{\rd \mu} -1 \right }_{L^2_\mu},
\end{equation}
where $\norm{ \xi K^n- \mu }_{\TV} \defeq  \sup_{A\in\Borel{\XX}} \vert \xi K^n(A) - \mu(A) \vert$ denotes the total variation distance between $\xi K^n= \Prob_{X_{n+1}}$ and $\mu$, $\xi\defeq \Prob_{X_1}$, and $\norm{ f }_{L^p_\mu}^p \defeq  \int_{\XX} \absval{ f }^p \, \rd \mu $ for $f\colon \XX \to \Reals$.
Furthermore, \citet{Ru12} shows that, for every $p>2$, there exists an explicit constant $c_p$ such that, for $f\in L^p(\mu)$ with $\norm{f}^{p}_{L^p_\mu}\leq 1$,
\[
	\ev{ \Absval{ \frac{1}{N} \sum_{n = 1}^{N} f(X_{n}) - \int_{\XX} f {\rd} \mu }^2 }
	\leq \frac{2}{N\cdot \gap_\mu(K)} +
	\frac{c_p 	\Norm{ \frac{\rd \xi }{\rd \mu} -1 }_{\infty}}{N^2\cdot\gap_\mu(K)^2} .
\]
There are also other non-asymptotic bounds that consider different convergence assumptions on the Markov chain and the underlying error criterion.
For details we refer to \citep{LaNi11,LaMiNi13,rudolf2015error,paulin2015concentration,Fan2021hoeffding} and the references therein.

In the following we briefly discuss two popular methods for the construction of a transition kernel $K$ that is reversible \wrt $\mu$.

\subsection{Metropolis--Hastings (MH) algorithms}

Probably the most popular method for constructing a transition kernel that is reversible \wrt $\mu$ is given by the \ac{MH} algorithm.
Given a Markov kernel $Q$ on $\XX$ which serves as a proposal mechanism, i.e.\ a `proposal kernel', and given a function $\alpha \colon \XX \times \XX \to [0,1]$ which provides acceptance probabilities and depends on $\mu$ and $Q$, a transition from a state $x$ to a state $y$ in the \ac{MH} algorithm proceeds as follows.

\begin{myAlgorithm}[Metropolis--Hastings]
\label{alg:Metropolis_Hastings}
	Let $Q$ be a proposal kernel and $\alpha \colon \XX \times \XX \to [0,1]$ be an acceptance probability function.
	Given the current state $x\in\XX$,
	one obtains the next state $y\in\XX$ as follows:
	\begin{enumerate}[label=(\arabic*)]
		\item Draw $X'\sim Q(x,\quark)$ and $U\sim \Uniform[(0,1)]$ independently and denote the realisations by $x'$ and $u$ respectively;
		\item If $u<\alpha(x,x')$, return $y \defeq x'$, otherwise return $y \defeq x$.
	\end{enumerate}
\end{myAlgorithm}
The algorithm above can be rewritten as a transition kernel:
\begin{equation}
\label{eq:MH_transition_kernel}
	K(x, \rd y)
	=
	\alpha(x,y) \PropKernel(x, \rd y) + r(x) \delta_x(\rd y),
	\qquad
	r(x) \defeq 1 - \int_{\XX} \alpha(x,y) \, \PropKernel(x, \rd y),
\end{equation}
where $\delta_x$ denotes the Dirac measure on $\XX$ at $x\in\XX$ and the function $r$ is called the `rejection probability'.
It remains to specify $\alpha$.
Let $\sigma^{+},\sigma^{-} \in \probson{\XX \times \XX}$ with
\begin{equation}
\label{eq:sigma_plus_minus}
\begin{aligned}
\sigma^{+} (\rd x, \rd x') & \defeq \PropKernel (x, \rd x') \mu (\rd x) , \\
\sigma^{-} (\rd x, \rd x') & \defeq \sigma^{+} (\rd x', \rd x) = \PropKernel (x', \rd x) \mu (\rd x'),
\end{aligned}
\end{equation}
and set
\begin{align*}
\alpha(x, x') & \defeq \min \left\{1 , \frac{\rd \sigma^{-}}{\rd \sigma^{+}} (x, x') \right\} .
\end{align*}
Then the transition kernel $K$ defined in \eqref{eq:MH_transition_kernel} is reversible \wrt $\mu$ \citep{Tierney1998}.
Of course, the Radon--Nikodym derivative $\frac{\rd \sigma^{-}}{\rd \sigma^{+}} (x, x')$ does not necessarily exist.
In the finite-dimensional Euclidean setting where $\XX = \Reals^{d}$, the derivative $\frac{\rd \sigma^{-}}{\rd \sigma^{+}}$ is often just the ratio of Lebesgue densities.
However, in infinite-dimensional spaces the absolute continuity $\sigma^{-}\ll\sigma^{+}$ requires the choice of a suitable proposal kernel $\PropKernel$.

\subsection{Slice sampling}

Suppose that a $\sigma$-finite reference measure $\mu_{0}$ on $\XX$ is given that satisfies $\mu \ll \mu_{0}$.
Additionally, we assume that the probability density function $\frac{\rd \mu}{\rd \mu_{0}}$ satisfies
\[
	\frac{\rd \mu}{\rd \mu_{0}}(x)
	\propto \exp( - \Phi(x)),
	\quad
	\mu_{0} \text{-a.e.\ }x\in\XX,
\]
for a measurable function $\Phi\colon \XX \to \Reals$ such that $\exp(-\Phi)$ is integrable \wrt $\mu_{0}$.
For some $s>0$, define
\begin{equation}
\label{eq:super_level_set}
\XX_s \defeq \set{ x\in\XX }{ \exp(-\Phi(x))\geq s },
\end{equation}
to be the super-level set of $\exp(-\Phi)$ to level $s$.
Let $\norm{ \exp(-\Phi)}_\infty \defeq  \mathop{\mu_0\text{-ess\,}\sup}\limits_{x\in\XX} \vert \exp(-\Phi(x))\vert$
and note that $\mu_{0}(\XX_s)\in (0,\infty)$ for $s\in(0,\norm{ \exp(-\Phi)}_\infty)$.
Here,
\[
	\mathop{\mu_0\textup{-ess\,sup}}\limits_{x\in\XX} f(x)
\]
denotes the essential supremum with respect to $\mu_0$ of $f \colon \XX \to \Reals$.
Define the probability measure $\mu_{0,s}\in\probson{\XX}$ by
\[
\mu_{0,s}(A) \defeq \frac{\mu_{0}(A\cap \XX_s)}{\mu_{0}(\XX_s)}, \quad A\in\Borel{\XX},
\]
that is, $\mu_{0,s}$ is the normalised restriction of $\mu_{0}$ to $\XX_s$.
In the idealised slice sampling algorithm, a transition from a state $x$ to a state $y$ proceeds as follows.

\begin{myAlgorithm}[Idealised slice sampling]
	\label{alg:ideal_ss}
Given the current state $x\in\XX$
one obtains the next state $y\in\XX$ as follows:
	\begin{enumerate}[label=(\arabic*)]
		\item Draw $S\sim\Uniform[(0,\exp(-\Phi(x)))]$ and let $s$ be the realisation.
		\item Draw $Y\sim \mu_{0,s}$ and let the state $y$ be the realisation of $Y$.
	\end{enumerate}
\end{myAlgorithm}
The corresponding transition kernel takes the form
\[
	K(x,\rd y) = \frac{1}{\exp(-\Phi(x))} \int_0^{\exp(-\Phi(x))} \mu_{0,s}(\rd y)\, \rd s,
\]
and it can be readily shown that $K$ is reversible \wrt $\mu$.

\begin{remark}
	In the case that $\XX=\Reals^d$ and $\mu_0$ is the Lebesgue measure, $\mu_{0,s}$ coincides with the uniform distribution on $\XX_s$.
	In that setting the corresponding method is known as simple (uniform) slice sampling and recent results \citep[Theorem~3.10]{Natarovskii2021} concerning the spectral gap indicate under which conditions robust convergence behaviour \wrt the dimension is present.
\end{remark}

However, the main issue with the idealised slice sampling algorithm is that the second step in \Cref{alg:ideal_ss} may be difficult to implement, because the implicit assumption of being able to draw samples from $\mu_{0,s}$ for an arbitrary level $s$ may be very restrictive.
Whenever one is not able to sample from a distribution exactly, one can try to use a suitable Markov chain step.
Namely, one substitutes the second step of \Cref{alg:ideal_ss} by performing a transition, depending on $x$ and $t$, which (at least) has stationary distribution $\mu_{0,s}$.
Such approaches are known as \emph{hybrid slice sampling} strategies; see e.g.\ \citep{latuszynski2014convergence} for some theory and comparison results.
There are several different methods in the literature that are feasible in finite-dimensional settings, and can --- from an algorithmic perspective --- be lifted to infinite-dimensional scenarios; see e.g.\ \citep{neal2003slice,Murray2010,li2020latent}.

\section{Two random walk-like \ac{MCMC} algorithms on $\Sphere^{d-1}$}
\label{sec:GeodesicZappa}

We describe now two \ac{MH} algorithms from the literature.
These algorithms use random walk-like proposals that are defined on general, finite-dimensional manifolds $\MM$.
We shall compare the algorithms described in \Cref{sec:reprojection_method} with these algorithms on a numerical example in \Cref{sec:numerical_illustrations}.
In the following paragraphs we first describe the corresponding \ac{MH} algorithm for a general manifold $\MM$ and then provide the particular algorithm for the case of sampling on a sphere $\MM = \Sphere^{d-1}$.
Throughout, we assume that $\MM$ admits a Hausdorff measure $\Hausdorff_{\MM}$ and that the target probability measure $\mu$ on $\MM$ is given by an unnormalised density $\rho$ \wrt $\Hausdorff_{\MM}$, that is,
\begin{equation}
	\label{equ:post_aux}
	\frac{\rd \mu}{\rd \Hausdorff_{\MM}}(\xx) \propto \rho(\xx), \qquad \xx \in \MM.
\end{equation}
When $\MM = \Sphere^{d-1}$ and $\mu$ is a posterior measure \wrt an \ac{ACG} prior $\mu_0 = \ACG(C)$ as in \eqref{eq:push_post},
\[
	\rho(\xx)
	=
	\frac{\exp(-\PPhi(\xx))}{\norm{ \xx }^d_C}.
\]

\paragraph{Geodesic random walk-\ac{MH} algorithm.}
Assume that $\Hausdorff_{\MM}(\MM)<\infty$.
Then the uniform measure $\Uniform(\MM)$ on $\MM$ is defined via $\Uniform(\MM)(A) \defeq \Hausdorff_{\MM}(A)/\Hausdorff_{\MM}(\MM)$ for $A \in \Borel{\MM}$.
The geodesic random walk as described in \citep{MangoubiSmith2018} yields a Markov chain $(\bar{X}_k)_{k\in\Naturals}$ on $\MM$ with $\Uniform(\MM)$ as its limit distribution.
For any $\xx \in \MM$, denote the tangent space at $\xx$ by $\mathcal{T}_{\xx}\MM$, and for any vector $v \in \mathcal{T}_{\xx}\MM$, denote by $\gamma_{\xx,v}$ the unique geodesic $\gamma_{\xx,v} \colon [0,\infty) \to \MM$ on $\MM$ that satisfies $\gamma_{\xx,v}(0) = \xx$ and $\gamma'_{\xx,v}(0) = v$, where $\gamma'$ denotes the first derivative.

Next, we present how a transition from a state $\xx$ to a state $\yy$ proceeds in the geodesic random walk algorithm.

\begin{myAlgorithm}[Geodesic random walk]
	Given the current state $\xx\in\MM$
	one obtains the next state $\yy\in\MM$ for fixed $t>0$ as follows:
	\begin{enumerate}[label=(\arabic*)]
		\item
		Draw from the uniform distribution
		on the unit sphere in $\mathcal{T}_{\xx}\MM$ and call the result $v$;

		\item
		Set $\yy = \gamma_{\xx, v}(t)$.
	\end{enumerate}
\end{myAlgorithm}

Hence, in order to sample approximately from $\mu$ as in \eqref{equ:post_aux}, we can employ the geodesic random walk kernel as a proposal kernel in an \ac{MH} algorithm.
The resulting acceptance probability $\alpha$ involves ratios of the unnormalised density $\rho$ in \eqref{equ:post_aux}.
\citet[Theorem~27]{goyal2019sampling} argue that the corresponding geodesic random walk transition kernel is reversible \wrt $\Uniform(\MM)$.
Therefore, the resulting algorithm for realising the Metropolised geodesic random walk on $\MM=\Sphere^{d-1}$ is given as in \Cref{alg:geodesic_MCMC}.

\begin{remark}\label{rem:Tx_ONB}
In order to sample from the uniform distribution on the unit sphere in $\mathcal{T}_{\xx}\MM$ we can use the tools outlined by \citet{Zappa2018} for the case of constrained $m$-dimensional manifolds embedded into $\Reals^d$, $d>m$, via
\begin{equation}\label{eq:M_constraint}
	\MM = \Set{x \in \Reals^d }{ q_i(x) =
	0 \quad \forall i = 1,\ldots,L }, \qquad L\geq d-m,
\end{equation}
where $q_i\colon \Reals \to \Reals$ are smooth functions.
In particular, an orthonormal basis (ONB) of $\mathcal{T}_{\xx}\MM$ can be  obtained by a QR decomposition of the Jacobian of $Q(x) \defeq  (q_1(x),\ldots,q_L(x))$ evaluated at $x=\xx$.
Let us denote such a basis by $u_1,\ldots, u_{m} \in \Reals^d$ and store the vectors in the columns of $U_{\xx} = [u_1 \mid \ldots \mid u_{m}] \in \Reals^{d \times m}$.
Drawing a sample $w$ from $\Uniform(\Sphere^{m-1})$ and setting $v = U_{\xx}w$ yields a sample according to the uniform distribution on the sphere in $\mathcal{T}_{\xx}\MM$.
If $\MM=\Sphere^{d-1}$, then there is only one constraint function, i.e.\ $q_1(x) = \norm{ x }^2 -1$.
Hence, the tangent space $\mathcal{T}_{\xx}\MM$ coincides with the orthogonal complement $\xx^\perp$ of $\sspan\{\xx\}$.
\end{remark}

\begin{algorithm}
  \caption{Geodesic random walk-\ac{MH} on $\Sphere^{d-1}$  \citep{MangoubiSmith2018}} \label{alg:geodesic_MCMC}
 \begin{algorithmic}[1]
  \STATE \textbf{Given:} time $t \in (0, \frac \pi2]$ and initial state $\xx_0 \in \Sphere^{d-1}$
  \FOR{$k \in\Naturals_0$}
	\STATE Compute ONB matrix $U_{{\xx}_k} \in \Reals^{d\times (d-1)}$ of $\xx_k^\perp$ according to \Cref{rem:Tx_ONB}
	\STATE Draw a sample $w_k$ from $\Uniform(\Sphere^{d-2})$ and set $v_k \defeq  U_{{\xx}_k}w_k$
	\STATE Set $\yy_{k+1} \defeq  \cos(t) \xx_k + \sin(t) v_k$
	\STATE Compute $a \defeq  \min\{1, \rho(\yy_{k+1}) / \rho(\xx_{k})\}$
	\STATE Draw a sample $u$ from $\Uniform([0,1])$
	\IF{$u\leq a$}
		\STATE Set $\xx_{k+1} = \yy_{k+1}$
	\ELSE
		\STATE Set $\xx_{k+1} = \xx_{k}$
	\ENDIF
  \ENDFOR
    \end{algorithmic}
\end{algorithm}

\begin{remark}[Ergodicity]
\label{rem:ergodicity}
In \citep[Theorem~7.1]{MangoubiSmith2018} it is shown that the geodesic random walk, i.e.\ the proposal kernel of \Cref{alg:geodesic_MCMC}, possesses a mixing time \wrt a Wasserstein distance that depends only on the positive curvature of the manifold $\MM$, given a suitably small integration time $t$.
Thus, in case of the sphere $\MM = \Sphere^{d-1}$, the mixing time is independent of the dimension $d$ for any $t \in (0, \frac \pi2]$.
However, this statement solely holds for the proposal chain.
As we will see in \Cref{sec:numerical_illustrations}, the \ac{MH} algorithm based on the geodesic random walk proposal shows deteriorating efficiency as $d\to\infty$.
Besides that, also the uniform ergodicity of \Cref{alg:geodesic_MCMC}, or the geodesic random walk itself, is not obvious and left open for future research.
\end{remark}

\paragraph{MH algorithm by \citet{Zappa2018}.}
The \ac{MH} algorithm presented in \citep{Zappa2018} shows some similarities with our method, in the sense that it first proposes a new state in the ambient Euclidean space which is then projected to the manifold $\MM$.
The manifold $\MM$ is assumed to be of the form \eqref{eq:M_constraint}.
Given a current state $\xx \in \MM$, we first draw a tangent vector $v \in \mathcal{T}_{\xx}\MM$, but this time \wrt an isotropic multivariate Gaussian measure $\Normal(0, s^2 I_{d-m})$.
Here, we can employ again the technique explained in \Cref{rem:Tx_ONB}, by using an ONB matrix $U_{\xx} \in \Reals^{d\times (d-m)}$ of $\mathcal{T}_{\xx}\MM$, drawing $w$ from $\Normal(0, s^2 I_{d-m})$, and defining the resulting sample $v= U_{\xx}w$.
We then consider $x \defeq \xx + v \in \Reals^d$, and project $x$ to some $\yy \in \MM$ by
\[
	\yy \defeq x + w
	\quad
	\text{ with suitable }
	\quad
	 w \in \left(\mathcal{T}_{\xx}\MM\right)^\perp.
\]
In general, computing $w$ and $\yy$ requires solving a nonlinear system;
see \citep[Equation~(2.5)]{Zappa2018}.
For $\MM = \Sphere^{d-1}$ the situation is rather easy: Since $\left(\mathcal{T}_{\xx}\MM\right)^\perp = \Span{\xx}$, $w = a \xx$ for some $a \in \Reals$ satisfying $\norm{\yy} = \norm{(1+a)\xx + v} = 1$.
Since $\xx \perp v$, if $\norm{ v }^2 \leq 1$, then
\[
	\yy = \sqrt{1 - \norm{ v }^2} \xx + v.
\]
If $\norm{ v }^2 > 1$, then $x$ cannot be projected back to the sphere along $\Span{\xx}$.
In this case, $x$ is rejected and the Markov chain remains at its current state, i.e.\ in the $k$th iteration it is realised as $\xx_{k+1} = \xx_k$.
In case of a successful proposal $\yy$ we still require a Metropolis step, where the correct acceptance probability for $\yy$ also requires the ingredients of the reverse move from $\yy$ to $\xx$.
That is, we require $\tilde{v} \in \mathcal{T}_{\yy}\MM$ such that
\[
	\xx = \yy + \tilde{v} + \tilde{w}, \qquad \tilde{w} \in (\mathcal{T}_{\yy}\MM)^\perp.
\]
The acceptance probability in the \ac{MH} algorithm targeting $\mu$ as in \eqref{equ:post_aux} is then given by
\[
	\alpha(\xx,\yy)
	=
	\min\left\{1, \frac{\rho(\yy) \, p(\yy,\tilde{v})}{\rho(\xx) \, p(\xx,v)} \right\},
\]
where $p(\yy,\tilde{v}) \propto \exp\left(-\frac{\norm{ U_{\yy}^\top \tilde{v}}^2}{2s^2} \right)$ denotes the proposal density for the tangential moves.
Since $U_{\yy}$ is orthonormal, $\norm{ U_{\yy}^\top \tilde{v}}=\norm{ \tilde{v}}$.
Moreover, for $\MM = \Sphere^{d-1}$, the vector $\tilde{v} \in \mathcal{T}_{\yy}\MM$ for going from $\yy = \sqrt{1 - \norm{ v }^2} \xx + v$ back to $\xx$ can be computed easily by projecting $\xx - \yy$ onto $\mathcal{T}_{\yy}\MM = (\yy)^\perp$, which yields
\[
	\tilde{v} = \xx - \yy - \innerprod{ \xx - \yy}{ \yy} \yy
	=
	\norm{ v }^2 \xx - \sqrt{1-\norm{ v }^2} v.
\]
In particular, since $\xx \perp v$, we obtain that $\norm{ \tilde{v} } = \norm{ v}$.
Hence, for $\MM = \Sphere^{d-1}$ the acceptance probability is just $\alpha(\xx,\yy) = \min\left\{1, \rho(\yy) / \rho(\xx)\right\}$.
We summarise the resulting \ac{MH} algorithm in \Cref{alg:Zappa_MCMC}.

\begin{algorithm}
  \caption{MH algorithm on $\Sphere^{d-1}$ \citep{Zappa2018}} \label{alg:Zappa_MCMC}
 \begin{algorithmic}[1]
  \STATE \textbf{Given}: step size $s>0$ and initial state $\xx_0 \in \Sphere^{d-1}$
  \FOR{$k \in\Naturals_0$}
	\STATE Compute ONB matrix $U_{{\xx}_k} \in \Reals^{d\times (d-1)}$ of $\xx_k^\perp$ according to \Cref{rem:Tx_ONB}
	\STATE Draw a sample $w_k$ of $\Normal(0, s^2 I_{d-1})$ and set $v_k \defeq  U_{{\xx}_k}w_k$
	\IF{$ \norm{ v } > 1$}
		\STATE Set $\xx_{k+1} = \xx_{k}$
	\ELSE
		\STATE Set $\yy_{k+1} \defeq   \sqrt{1 - \norm{ v }^2} \xx_k + v_k$
		\STATE Compute $a \defeq  \min\{1, \rho(\yy_{k+1}) / \rho(\xx_{k})\}$
		\STATE Draw a sample $u$ of $\Uniform[0,1]$
		\IF{$u\leq a$}
			\STATE Set $\xx_{k+1} = \yy_{k+1}$
		\ELSE
			\STATE Set $\xx_{k+1} = \xx_{k}$
		\ENDIF
	\ENDIF
  \ENDFOR
    \end{algorithmic}
\end{algorithm}

\begin{remark}[Ergodicity]
As stated in \citep[Section~2.1]{Zappa2018}, their proposed \ac{MH} algorithm yields uniform ergodicity for compact $\MM$ and continuous $\rho$.
This follows by standard arguments for \ac{MH} algorithms with continuous proposal densities bounded away from zero on compact state spaces, see, e.g.\ \cite[Example 15.3.2]{DoucEtAl2018}.
In particular, \Cref{alg:Zappa_MCMC} yields uniformly ergodic Markov chains under the conditions of \Cref{theo:Uniform_ergodic}.
\end{remark}

\section{Counterexample for the Markovianity of $T(X_k)$}
\label{sec:image_of_MC_is_not_MC}
    Let $\HH = \Reals^d$, $d>1$, and $\RadProjToSphere$ be the radial projection onto the unit sphere in $\HH$ as in \eqref{eq:radial_projection_map_T}, where we have fixed $\bar{z}=e_{d}=(0,\ldots,0,1)^\top$.
	Moreover, let $(X_n)_{n\in\Naturals}$ denote a Markov chain on $\HH$ with transition kernel $K$.

	If $(\RadProjToSphere (X_n))_{n\in\Naturals}$ were again a Markov chain, then we would have for the ``upper half'' of the sphere $H_d \defeq \set{ x = (x_1,\ldots,x_d)^\top \in \Sphere }{ x_d \geq 0 }$ that
	\[
		\Prob(\RadProjToSphere(X_2) \in H_d \ | \ \RadProjToSphere(X_1) = e_d)
		=
		\Prob(\RadProjToSphere(X_2) \in H_d \ | \ \RadProjToSphere(X_1) = e_d, \RadProjToSphere(X_0) = e_d),
	\]
	or, $\Prob(\RadProjToSphere(X_{n+2}) \in H_d\ | \ \RadProjToSphere(X_{n+1}) = e_d) = \Prob(\RadProjToSphere(X_{n+2}) \in H_d\ | \ \RadProjToSphere(X_{n+1}) = e_d, \RadProjToSphere(X_n) = e_d)$ for any $n\in\Naturals$.
	By definition of $\RadProjToSphere$, this is equivalent to
	\[
		\frac{\Prob(X_{2,d} > 0, X_{1,d} >0)}{\Prob(X_{1,d} > 0)}
		=
		\frac{\Prob(X_{2,d} > 0, X_{1,d} > 0, X_{0,d} > 0)}{\Prob(X_{1,d} > 0, X_{0,d} > 0)},
	\]
	where $X_n = (X_{n,1},\ldots,X_{n,d})^\top$ denotes the state vector of the Markov chain.
	Consider now a Markov chain $(X_n)_{n\in\Naturals}$ with initial distribution $X_0 \sim \Normal(0,I)$ and Gaussian random walk transition kernel $K(x, \quark) = \Normal(x, s^2 I)$ for a fixed step size $s > 0$.
	Then each component of the states is a Markov chain independent of the other components.
	Hence,
	\[
		\frac{\Prob(X_{2,d} > 0, X_{1,d} > 0)}{\Prob(X_{1,d} > 0)}
		=
		\frac{\Prob(W_0 + s W_1 + s W_2 > 0, W_0 + s W_1 >0 )}{\Prob(W_0 + s W_1 >0 )}
	\]
	for $W_0, W_1, W_2 \iidsim \Normal(0, 1)$ and, analogously,
	\[
		\frac{\Prob(X_{2,d} > 0, X_{1,d}, X_{0,d} > 0)}{\Prob(X_{1,d} > 0, X_{0,d} > 0)}
		=
		\frac{\Prob(W_0 + s W_1 + s W_2 > 0, W_0 + s W_1 >0, W_0 > 0 )}{\Prob(W_0 + s W_1 > 0, W_0 >0 )}.
	\]
	The expressions for these probabilities are difficult to evaluate exactly.
	Therefore, we perform a Monte Carlo integration using $10^8$ samples and obtain the following estimates for $s=1$:
	\begin{equation*}
	\begin{aligned}
		\frac{\Prob(W_0 + s W_1 + s W_2 > 0, W_0 + s W_1 >0 )}{\Prob(W_0 + s W_1 >0 )}
		& \approx
		0.8041,
		\\
		\frac{\Prob(W_0 + s W_1 + s W_2 > 0, W_0 + s W_1 >0, W_0 > 0 )}{\Prob(W_0 + s W_1 > 0, W_0 >0 )}
		& \approx
		0.8333.
		\end{aligned}
	\end{equation*}
	We now show that the above observation is not restricted to the random walk transition kernel.
	Consider a stationary Markov chain generated by the \ac{pCN} proposal kernel $Q(x, \quark) = \Normal(\sqrt{1-s^2}x, s^2 I)$ and initial distribution $X_0 \sim \Normal(0,I)$.
	We again obtain independent scalar Markov chains in each component:
	\[
		\frac{\Prob(X_{2,d} > 0, X_{1,d} > 0)}{\Prob(X_{1,d} > 0)}
		=
		\frac{\Prob( (1-s^2)W_0 + s\sqrt{1-s^2} W_1 + s W_2 > 0, \sqrt{1-s^2}W_0 + s W_1 >0 )}{\Prob(\sqrt{1-s^2}W_0 + s W_1 >0 )}
	\]
	as well as
	\begin{align*}
		& \frac{\Prob(X_{2,d} > 0, X_{1,d}, X_{0,d} > 0)}{\Prob(X_{1,d} > 0, X_{0,d} > 0)} \\
		& \quad =
		\frac{\Prob( (1-s^2)W_0 + s\sqrt{1-s^2} W_1 + s W_2 > 0, \sqrt{1-s^2}W_0 + s W_1 >0, W_0 > 0 )}{\Prob(\sqrt{1-s^2}W_0  + s W_1 > 0, W_0 >0 )},
	\end{align*}
	where $W_0, W_1,W_2 \iidsim \Normal(0, 1)$.
	Monte Carlo integration with the stationary \ac{pCN}-generated Markov chain, $s = 0.5$, and $10^8$ samples yields
	\[
		\frac{\Prob(X_{2,d} > 0, X_{1,d} > 0)}{\Prob(X_{1,d} > 0)}
		\approx
		0.8333
		,\qquad
		\frac{\Prob(X_{2,d} > 0, X_{1,d}, X_{0,d} > 0)}{\Prob(X_{1,d} > 0, X_{0,d} > 0)}
		\approx
		0.8620.
	\]
	Thus, these Markov chains serve as counterexamples to the claim that $(\RadProjToSphere(X_n))_{n\in\Naturals}$ is again a Markov chain.

\end{document}